\documentclass[11pt]{article}
\usepackage{mathrsfs}
\usepackage{amssymb}
\usepackage{amsmath}
\usepackage{amsthm}
\usepackage{txfonts}
\usepackage{cite}
\usepackage{mathrsfs}
\makeatletter
\def\LaTeX{\leavevmode L\raise.42ex
    \hbox{\kern-.3em\size{\sf@size}{0pt}\selectfont A}\kern-.15em\TeX}
\makeatother

\newcommand{\BibTeX}{{\rm B\kern-.05em{\sc
          i\kern-.025emb}\kern-.08em\TeX}}

\makeatletter
\def\@currentlabel{2.1}\label{e:dispaa}
\def\@currentlabel{2.21}\label{e:dispau}
\def\@currentlabel{2.22}\label{e:dispav}
\def\@currentlabel{2.23}\label{e:dispaw}
\def\@currentlabel{2.24}\label{e:dispax}
\def\theequation{\thesection.\@arabic\c@equation}
\makeatother

\newcounter{mnotecount}[section]

\newcommand{\rmnote}[1]{}

\newtheorem{Theorem}{Theorem}[section]
\newtheorem{Lemma}{Lemma}[section]

\newtheorem{Definition}{Definition}[section]

\renewcommand{\theequation}{\thesection.\arabic{equation}}

\date{}

\newcommand{\beq}{\begin{eqnarray}}
\newcommand{\eeq}{\end{eqnarray}}
\newcommand{\beqno}{\begin{eqnarray*}}
\newcommand{\eeqno}{\end{eqnarray*}}
\newcommand{\be}{\begin{equation}}
\newcommand{\ee}{\end{equation}}

\newcommand{\B}{\mathbb B}

\newcommand{\R}{\mathbb{R}}
\newcommand{\aint}{-\!\!\!\!\!\!\int}

\def \d2{\Delta_{+}^2}
\begin{document}

\title{Forward self-similar solutions to the viscoelastic Navier-Stokes equation with damping}

\author{Baishun Lai, Junyu Lin, Changyou Wang}
\maketitle
\date{}
\begin{abstract}
Motivated by \cite{JS}, we prove that there exists a global, forward self-similar solution to the viscoelastic Navier-Stokes equation with damping, that is smooth for $t>0$,  for any initial data  that is homogeneous of degree $-1$.
\end{abstract}

\noindent
{\small {\bf Mathematics Subject Classification (2000):} 76B03, 35Q31.

}

\setcounter{section}{0}
\setcounter{equation}{0}
\section{Introduction }

The Oldroyd model for an incompressible, (linear) viscoelastic fluid in $\mathbb R^3$ 
is given by the following system of equations:
\begin{align}\label{E1.1+}  
\left\{
\begin{array}{llll}
\ \ \partial_{t}u+(u\cdot\nabla)u=\mu\Delta u-\nabla p+\nabla\cdot FF^{t},\\
\partial_{t}F+(u\cdot\nabla)F=\nabla uF,\\
\qquad\qquad\nabla\cdot u=0,
\end{array}\right.\  \ \  \ \ \ \mbox{in}\ \ \mathbb R^3\times (0,\infty)
\end{align}
where $u:\R^3\times(0,+\infty)\rightarrow \R^3$   represents the fluid velocity, $F:\R^3\times(0,+\infty)\rightarrow \R^{3\times3}$ represents the local deformation tensor of the fluid,  $p:\R^3\times(0,+\infty)\rightarrow \R$ represents the fluid pressure, and $\mu>0$ is the viscosity constant.
The $i^{th}$-component of $\nabla\cdot FF^{t}$ equals to $\nabla_{j}(F_{ik}F_{jk})$. 
The system \eqref{E1.1+} is also called as the viscoelastic Navier-Stokes equation.
A straightforward calculation implies that for any smooth solution $(u, p, F)$ to \eqref{E1.1+}, that decays sufficiently rapid 
near  infinity,  the energy satisfies
\begin{equation}\label{I1.1}
\frac{d}{dt}\Big\{\frac12\int_{\R^3}(|u|^2+|F|^2)\, dx\Big\}=-\mu\int_{\R^3}|\nabla u|^2\, dx,\ t>0.
\end{equation}
There have been several interesting works on the initial value problem of \eqref{E1.1+} by \cite{LL, LZ, HW} 
asserting both the existence of short time smooth solution and the global existence of smooth solution
for small initial data.  For large (rough) initial data, the global existence of weak solutions to \eqref{E1.1+}
has recently been achieved by \cite{HL1, HL2} in dimension two, but remains open in dimension three.
In pursuing global weak solutions of \eqref{E1.1+}, the authors of \cite{LL} proposed the following
viscoelastic Navier-Stokes equations with damping approximating \eqref{E1.1+}:
\begin{align}\label{E1.2+}  
\left\{
\begin{array}{llll}
\ \ \partial_{t}u+(u\cdot\nabla)u=\mu\Delta u-\nabla p+\nabla\cdot FF^{t},\\
\partial_{t}F+(u\cdot\nabla)F=\nu\Delta F+\nabla uF,\\
\qquad\qquad\nabla\cdot u=0,
\end{array}\right. \qquad\qquad \mbox{in}\ \ \mathbb R^3\times (0,\infty)
\end{align}
where $\nu>0$ is a damping constant.  Following the scheme by \cite{C} on the Navier-Stokes equation, 
it is not hard to establish the existence of global weak solutions $(u_\nu, p_\nu, F_\nu)$ to \eqref{E1.2+}.  It
is a challenging open problem to show that $(u_\nu, p_\nu, F_\nu)$ converges to a global weak solution
of \eqref{E1.1+} as $\nu\rightarrow 0$.  Nevertheless, the system \eqref{E1.2+} itself is an interesting 
system that deserves to be studied, besides that better understanding of \eqref{E1.2+} may be useful for \eqref{E1.1+}. 
We would like to mention a few basic facts on smooth solutions $(u,p, F)$ of \eqref{E1.2+}:
\begin{itemize}
\item[(i)] If $(u,p,F)$ has sufficiently rapid decay at infinity, then 
\begin{equation}\label{I1.2}
\frac{d}{dt}\Big\{\frac12\int_{\R^{3}}(|u|^{2}+|F|^{2})\, dx\Big\}
=-\int_{\R^{3}}(\mu|\nabla u|^{2}+\nu |\nabla F|^{2})\,dx,\ \ t>0.
\end{equation}
\item[(ii)] If $\nabla\cdot F=0$ at $t=0$, then $\nabla\cdot F=0$ for all $t>0$. This follows by taking divergence of
\eqref{E1.2+}$_2$ and applying \eqref{E1.2+}$_3$, which yields (see also \cite{H} and \cite{LL}):
$$
\partial_t(\nabla\cdot F)+(u\cdot\nabla)(\nabla\cdot F)=\mu\Delta (\nabla\cdot F).
$$
Here $(\nabla\cdot F)^{j} =\nabla_{i}F_{ij}$ denotes the divergence of $j^{th}$-column vector of $F$.
\item[(iii)] Under the condition $\nabla\cdot F=0$, it follows that 
$$
(\nabla\cdot FF^{t})^i=\partial_{j}F_{jk}F_{ik}+F_{jk}\partial_{j}F_{ik}=((F_{k}\cdot\nabla) F_{k})^i,
$$
where $F_k$ denotes the $k^{th}$-column vector of $F$, so that \eqref{E1.2+}
is closely related to the incompressible MHD system.
\end{itemize}

Similar to the Navier-Stokes equation,  
\eqref{E1.2+} is invariant under both translations and scalings: For any $z_0=(x_0,t_0)
\in\R^3\times (0,\infty)$ and $\lambda>0$,  if $(u,p, F)$ is a solution of \eqref{E1.2+},  then $
\big(u_{z_0,\lambda}, p_{z_0, \lambda}, F_{z_0,\lambda}\big)$ is also a solution of \eqref{E1.2+}, where
\[\begin{cases}
u_{z_0,\lambda}(x,t):=\lambda u(x_0+\lambda x,t_0+\lambda^2 t),\\
p_{z_0,\lambda}(x,t):=\lambda^2 p(x_0+\lambda x, t_0+\lambda^2 t),\\
F_{z_0, \lambda}(x,t):=\lambda F(x_0+\lambda x, t_0+\lambda^2 t).
\end{cases} \]
A solution $(u, p, F)$ to \eqref{E1.2+} is called to be self-similar,
if $$u=u_{(0,0), \lambda},\ p=p_{(0,0),\lambda}, \ F=F_{(0,0), \lambda},
\ \mbox{for every}\ \lambda>0.$$
In this case, the value of $(u, p, F)$ is determined by that at time $t=1$:
\begin{equation}\label{NE1.3}
u(x,t)=\frac{1}{\sqrt{t}}U(\frac{x}{\sqrt{t}}), \ p(x,t)=\frac{1}{t}P(\frac{x}{\sqrt{t}}),
\ F(x,t)=\frac{1}{\sqrt{t}}G(\frac{x}{\sqrt{t}}),
\end{equation}
where $(U,P, G)(y)=(u, p, F)(y,1)$. 

By simple calculations, the profile $(U, P, G)$ of a self-similar solution satisfies in $\R^3$:
\begin{align}\label{E1.3}
\left\{
\begin{array}{llll}
-\Delta U-\frac{1}{2}U-\frac{1}{2}(x\cdot\nabla) U+(U\cdot\nabla) U-(G_{k}\cdot\nabla)G_{k}+\nabla P=0,\\\\
{\rm{div}} U=0,\ {\rm{div}} G=0\\\\
-\Delta G_{j}-\frac{1}{2}G_{j}-\frac{1}{2}(x\cdot \nabla) G_{j}+(U\cdot\nabla)G_{j}-(G_{j}\cdot\nabla)U=0, \ j=1, 2, 3.
\end{array}
\right.
\end{align}
Here $G_j$ stands for the $j^{\rm{th}}$-column of $G$.

The initial condition of \eqref{E1.2+} is given by
 \begin{equation} \label{IC1.1}
 u(x,0)=u_{0}(x),\ F(x,0)=F_{0}(x), \ x\in\R^3,
 \end{equation}
for some given functions $u_0:\R^3\to \R^3$, with $\nabla\cdot u_0=0$, and $F_0:\R^3\to\R^{3\times 3}$,
with $\nabla\cdot F_0=0$.

When $(u, F)$ is self-similar, then $(u_{0},F_{0})$ must be homogeneous of degree $-1$, i.e., 
$$ (u_0(x), F_0(x))=\frac{1}{|x|}\Big(u_0\big(\frac{x}{|x|}\big), F_0\big(\frac{x}{|x|}\big)\Big), \ x\in\R^3\setminus\{0\}.$$
Thus it is natural to assume
\begin{equation}\label{E1.4}
|u_{0}(x)|+|F_{0}(x)|\leq\frac{C_{*}}{|x|},\ \  x\in \R^{3}\setminus\{0\},
\end{equation}
for some constant $C_{*}>0$ and look for self-similar solutions ($u, F$) satisfying
$$
|u(x,t)|+|F(x,t)|\leq\frac{C(C_{*})}{|x|}, \  x\in \R^{3}\setminus\{0\}, \ \  \mbox{or}\ \ \ \|u\|_{L^{3,\infty}}+\|F\|_{L^{3,\infty}}\leq C(C_{*}).
$$
Here $L^{q,r}, 1\leq q, r\leq \infty$, denotes the ($q,r$)-Lorentz space.

Concerning the Navier-Stokes equation, the existence and uniqueness of self-similar solutions with sufficiently small $C_{*}$ were established via the contraction mapping argument, see \cite{C1,C2} for details. For large $C_{*}$, Jia and \v{S}ver\'{a}k in their recent important work \cite{JS} have constructed a self-similar solution for every
$u_0$ that is homogeneous of degree $-1$ and locally H\"{o}lder continuous. The crucial ingredients in \cite{JS} are the
 ``local-in space near the initial time'' regularity techniques that ensure priori estimates of self-similar solutions.  Based on it, they showed the existence of global self-similar solutions to the Navier-Stokes equation by Leray-Schauder's degree theorem (see \cite{M}). 
The result of \cite{JS} was subsequently extended by Tsai \cite{T} and Bradshaw-Tsai \cite{BT} to forward discretely self-similar solutions, and Korobkow-Tsai \cite{KT} to forward self-similar solutions in the half space.

In this paper, we aim to extend the ideas by \cite{JS} on the Navier-Stokes equation to establish the existence of self-similar solutions of \eqref{E1.2+} for initial data ($u_{0}, F_{0}$) satisfying \eqref{E1.4} with large $C_{*}$. Our main theorem is

\begin{Theorem}\label{T1.1}
Let $u_{0}, F_{0}\in C_{loc}^{\gamma}\big(\R^{3}\setminus \{0\}\big)$ for some $\gamma\in (0,1)$, 
with $\mbox{div}\ u_{0}=0$ and $\mbox{div}\ F_0=0$,  be homogeneous of degree $-1$. 
Assume (\ref{E1.4}) holds, then the system (\ref{E1.2+}) and (\ref{IC1.1}) has at least one 
self-similar solution $(u,p, F)$ that is smooth in $\R^3\times(0,\infty)$  and locally H\"{o}lder continuous 
in $\R^3\times[0,\infty)\backslash\{(0,0)\}.$
\end{Theorem}
Since the values of $\mu$ and $\nu$ in \eqref{E1.2+} don't play any role in this paper,  we will assume throughout
this paper that $$\mu=\nu=1.$$
We take a slightly different approach from \cite{JS} to prove Theorem \ref{T1.1} by applying
the Leray-Schauder fixed point theorem (see \cite{G} Theorem 11.6). To this end, we introduce
a family of viscoelastic Navier-Stokes equations with damping as follows. For $0\le \sigma\le 1$,
consider
$$
\left\{
\begin{aligned}
\partial_{t}u-\Delta u+\nabla p=-\sigma[\nabla\cdot(u\otimes u)-\nabla\cdot FF^{t}],\\
\partial_{t}F-\Delta F=\sigma[\nabla u F-(u\cdot\nabla)F],\\
\nabla\cdot u=0,\ \nabla\cdot F=0\\
\end{aligned}\ \right. \qquad\qquad \ \ \mbox{in}\ \ \R^{3}\times (0,\infty)
\eqno{(1.3)_{\sigma}}
$$
It is clear that (1.3)$_\sigma$ reduces to the Stokes equation for $u$ and the heat equation for $F$ when $\sigma=0$,
and becomes \eqref{E1.1+} when $\sigma=1$. 
We will translate  the problem $(1.3)_{\sigma}$ and \eqref{IC1.1} into the fixed point problem:
$$
w=T(w,\sigma): X\to X,\ \ \  0\leq\sigma\leq1,
$$
where $X$ is a suitable Banach space, and $w=(v, H)$ is defined through
$$
u(x,t)=(e^{t\Delta}u_{0})(x)+v(x,t);\ \ F(x,t)=(e^{t\Delta}F_{0})(x)+H(x,t),
$$
with ($u, F$) and $(v, H)$ being self-similar solutions of (1.3)$_\sigma$ and (\ref{E3.9}) respectively.
To prove $T(\cdot,1)$ has a fixed point $w\in X$, we will verify that\\
(i) $T: X\times[0,1]\to X$ is a  compact operator and $T(w,0)=0$ for all $w\in X$.\\
(ii) There exists a constant $C$ such that
$$
\|v\|_{X}\leq C,
$$
for all $(v,\sigma)\in X\times[0,1]$ satisfying $v=T(v,\sigma)$. 

The main part of this paper is to verify the second condition. To achieve this, we will extend the ``local-in space near the initial time'' regularity technique, first developed by  \cite{JS} on the Navier-Stokes equation, to the system $(1.3)_{\sigma}$ or more precisely
\eqref{E2.1} below.  We will first prove $(v, H)$ is H\"{o}lder continuous under some smallness condition (the so called $\epsilon$-regularity criteria), then use this ``$\epsilon$- regularity" theorem to obtain the local in space near the
initial time smoothness for the so called local Leray weak solutions, from which we can establish the priori estimate of $(v, H)$. Finally, by the Leray-Schauder's fixed point theorem, we obtain the existence of global
forward self-similar solutions to \eqref{E1.2+}.

\bigskip
\noindent{\bf Notations.}  For a better presentation,  we list some notations here. 
\begin{align*}
&z=(x,t), \ z_{0}=(x_{0},t_{0}),\ \ \ \ \ \ \ \ \ \ \ \ \ \ \ \ \ \ \ \ \ \  \B_r(x_0)=\Big\{\ x\in R^3:\  |x-x_0|<r\ \Big\},\\
&Q_r(z_0)=\B_r(x_0)\times(t_0-r^2,t_0),\ \ \ \ \ \ \ \ \ \ \ \  \aint_{Q_r(z_0)} f \,dz=\frac{1}{|Q_r(z_0)|}\int_{Q_r(z_0)} f\,dz,\\
&\big(u\big)_{z_0,r}=\aint_{Q_r(z_0)}u dz,\ \ \ \ \ \ \ \ \ \ \  \ \ \ \ \ \ \ \ \ \ \ \ \ \ \  \ \ \big(p\big)_{x_0,r}(t)=\aint_{B_r(x_0)}p(x,t)dx,\\
& \B_r=\B_r(0), \ \ \ \ Q_r=Q_r(0),
\ \ \ \ \ \ \ \ \ \ \ \ \ \ \ \ \ \B=\B_{1}, \ \ \ \ \ \ \ \
 Q=Q_1, \\ & (u)_{r}=(u)_{0,r}, \ \ (p)_{r}(t)=(p)_{0, r}(t),\ \ \ \ \ \ \ \  \ {\bf H}=\Big\{\ u\in L^{2}(\R^3,\R^3):\  {\rm{div}} u=0\ \Big\},\\
&{\bf V}=\Big\{\ u\in H^{1}(\R^3,\R^3): \ {\rm{div}} u=0\ \Big\}, \ \ \ \ \ \ \ v\otimes w=(v^iw^j)_{1\le i, j\le 3},\\
&A \cdot B=A_{ij}B_{ij}, \ |A|=\sqrt{A \cdot A}, \ Au=(A_{ij}u^{j})\in\R^{3}, \ \mbox{for}\ A, B\in\mathbb R^{3\times 3}
\ \mbox{and}\ u\in\R^3,\\
& \nabla u=\frac{\partial u}{\partial x_{j}}=(u_{ij})\ \mbox{for}\ u:\R^3\to\R^3.
\end{align*}

\setcounter{section}{1}
\setcounter{equation}{0}
\section{$\epsilon$-regularity}
Our goal in this section is to prove an $\epsilon$-regularity criteria, similar to that by Caffarelli-Kohn-Nirenberg on the Navier-Stokes equation \cite{C},  for a family of generalized viscoelastic Navier-Stokes equations with damping.
First we will introduce the perturbations of (1.3)$_\sigma$ as follows.  For any $\sigma\in [0,1]$, consider
\begin{align}\label{E2.1}
\left\{
\begin{array}{llll}
\partial_{t}v-\Delta v+\nabla p=-\sigma\big[v\cdot\nabla a+a\cdot\nabla v+v\cdot\nabla v\big]
+\sigma\nabla\cdot(HH^t+MH^t+HM^t),\\
\ \ \ \ \ \ \ \partial_{t} H-\Delta H=-\sigma\big[a\cdot\nabla H+v\cdot\nabla H+v\cdot\nabla M\big]
+\sigma\big[\nabla a H+\nabla v H+\nabla v M\big],\\
\ \ \ \ \ \ \ \ \ \ \ \ \ \ \ \ \mbox{div} v=0, \ \mbox{div} H=0, \\
\end{array}
\right.
\end{align}
in $Q$, where $a\in L^m(Q,\R^3)$, with $\nabla\cdot a=0$,
and $M\in L^{m}(Q,\R^{3\times 3})$, with $\nabla\cdot M=0$, for $m>5$ 
are given\footnote{For applications in later sections, 
($a,M$) will be chosen to be a mild solution of (1.3)$_\sigma$ under suitable initial data.}. 

We will study the regularity of suitable weak solutions of the system \eqref{E2.1} under a smallness condition.
For this, we will first recall the definition of weak solutions of the system \eqref{E2.1}, which
is consistent with the notion introduced by \cite{C} and \cite{JS} on the Navier-Stokes equation.

\begin{Definition}
A triple of functions $(v, p, H):Q\to\R^3\times\R\times\R^{3\times 3}$ is called a suitable weak solution to 
the system \eqref{E2.1}, if the following statements hold:
\begin{itemize}
\item[(a)]
$v, H\in L_{t}^{\infty}L_{x}^{2}(O)\cap L_{t}^{2}{\bf V}(O)\ \mbox{for any}\ O\subset Q,\  p\in L_{loc}^{\frac{3}{2}}(Q),$
and $(v, H)\in  L_{loc}^{3}(Q)$.
\item[(b)] $(v, H, p)$ satisfies the system \eqref{E2.1} on $Q$ in the sense of distributions.
\item[(c)] $(v,H, p)$ satisfies the following local energy inequality
\begin{align}\label{E2.2+}
\nonumber&\frac{d}{dt}\int_{\B}\phi\big(|v|^2+|H|^{2}\big)+2\int_{\B}\phi\big(|\nabla v|^{2}+|\nabla H|^{2}\big)\\
\nonumber&\leq\int_{\B}(\phi_{t}+\Delta\phi)(|v|^{2}+|H|^{2})-2\sigma v\otimes\nabla \phi\cdot HH^{t}\\
&+\int_{\B}\big[\sigma(|v|^{2}+|H|^{2})(v+a)+2pv\big]\cdot \nabla \phi\nonumber\\
\nonumber&+2\sigma\int_{\B}\big(a\otimes v-MH^t- HM^t\big)\cdot\big(\nabla v \phi+v\otimes\nabla \phi\big)\\
&+2\sigma\sum\limits_{i,j,k=1}^{3}\int_{\B}\big(v_kM_{ij}-a_iH_{kj}-v_iM_{kj}\big)\cdot \big(\nabla_k H_{ij}\phi+H_{ij}\nabla_k \phi\big),
\end{align}
for any $\phi\in C_{0}^{\infty}(\B\times (-1,t])$, with $\phi\geq0$.
\end{itemize}
\end{Definition}

Following the scheme in \cite{C}, it is not difficult to prove the existence of {\em suitable weak solutions} to the system
\eqref{E2.1} under suitable initial and boundary conditions. Note also that by the local energy inequality above and the known
multiplicative inequalities, any suitable weak solution ($v, p, H$) to \eqref{E2.1} satisfies
$$
v, H\in W_{\frac{5}{4},loc}^{2,1}(Q),\ \mbox{and}\ \ p\in W_{\frac{5}{4},loc}^{1,0}(Q).
$$

Extending the arguments by \cite{ESS, LS, L, JS} on the Navier-Stokes equation, we will prove the main theorem of this section.

\begin{Theorem}\label{Th2.1}
Let $(v,H,p)$ be a suitable weak solution to the system \eqref{E2.1} in $Q$, with $a\in L^{m}(Q, \R^3), \mbox{div}\ a=0$,
and $M\in L^m(Q,\R^{3\times 3})$, $\mbox{div} M=0$,  for some $m>5$. Then there exist
$\epsilon_{0}=\epsilon_{0}(m)>0$ and $\alpha_0=\alpha_0(m)\in (0,1)$, independent of $0\le\sigma\le 1$, such that if
\begin{align}\label{Cond2.1}
\Big(\aint_{Q}|v|^{3}\Big)^{\frac{1}{3}}+\Big(\aint_{Q}|H|^{3}\Big)^{\frac{1}{3}}+\Big(\aint_{Q}|p|^{\frac{3}{2}}\Big)^{\frac{2}{3}}
+\Big(\aint_{Q}|a|^{m}\Big)^{\frac{1}{m}}+\Big(\aint_{Q}|M|^{m}\Big)^{\frac{1}{m}}\leq \epsilon_{0},
\end{align}
then $(v,H)$ is H\"{o}lder continuous in $Q_{\frac{1}{2}}$ with exponent $\alpha_0$, and
\begin{equation}\label{Conc2.2}
\big\|v\big\|_{C_{par}^{\alpha_0}(Q_{\frac{1}{2}})}+\big\|H\big\|_{C_{par}^{\alpha_0}(Q_{\frac{1}{2}})}\leq C(m,\epsilon_{0}).
\end{equation}
\end{Theorem}


For the next lemma, we introduce the quantity
\begin{align*}
&Y\big(v,H,p, Q_{R}(z_{0})\big):=\Big(\aint_{Q_R(z_0)}|v-(v)_{z_0,R}|^{3}\Big)^{\frac{1}{3}}
+\Big(\aint_{Q_R({z_{0}})}|H-(H)_{z_{0},R}|^{3}\Big)^{\frac{1}{3}}\\
&\qquad\qquad\qquad\ \ \ \ \ \ \ \ +R\Big(\aint_{Q_R(z_{0})}|p-(p)_{x_{0},R}(t)|^{\frac{3}{2}}\Big)^{\frac{2}{3}}.
\end{align*}
\vskip0.2in

\begin{Lemma}\label{L2.2}
For $m>5$, suppose that $\mbox{div}\ a=0$,  $\mbox{div}\ M=0$, and
$$
\|a\|_{L^{m}(Q)}+\|M\|_{L^{m}(Q)}\leq \eta;\ \ |(v)_{1}|+|(H)_{1}|\leq L,
$$
for some small absolute number $\eta>0$ and some positive number $L$. Then for
 any $\theta\in (0,\frac{1}{3})$, there exist an $\epsilon=\epsilon(\theta,L,m)>0, C=C(L,m)>0$, and $\alpha=\alpha(m)>0$ such that if ($v, H, p$) is a suitable weak solution to \eqref{E2.1} with $0\le \sigma\le 1$, satisfying
\begin{equation}\label{Cond2.3}
Y(v,H,p, Q)+\Big\{|(v)_{1}|+|(H)_{1}|\Big\}\Big\{\aint_{Q}\Big(|a|^{m}+|M|^{m}\Big)\Big\}^{\frac{1}{m}}<\epsilon,
\end{equation}
then
\begin{align}\label{Conc2.4}
Y(v,H,p, Q_{\theta})\leq C\theta^{\alpha}\Bigg\{Y(v,H,p, Q)+\Big[|(v)_{1}|+|(H)_{1}|\Big]\Big[\aint_{Q}|a|^{m}+|M|^{m}\Big]^{\frac{1}{m}}\Bigg\}.
\end{align}
\end{Lemma}
\begin{proof}\ Suppose the lemma were false. Then there exist $\sigma_i\in [0,1]$ and a sequence of suitable weak
solutions  $(v_{i}, p_{i}, H^{i} )$ of \eqref{E2.1}, with $\sigma=\sigma_i$, $a=a_{i}$, and $M=M^{i}$, such that
\begin{align*}
&\|a_{i}\|_{L^{m}(Q)}+\|M^{i}\|_{L^{m}(Q)}\leq \delta,\ \ |(v_{i})_{1}|+|(H^{i})|_{1}\leq L,\ \ \nabla\cdot a_{i}=0,\\
&\epsilon_i\equiv Y(v_{i},H^{i},p_{i}, Q)+\Big\{|(v_{i})_{1}|+|(H^{i})_{1}|\Big\}\Big\{\aint_{Q}|a_{i}|^{m}+\aint_{Q}|M^{i}|^{m}\Big\}^{\frac{1}{m}}\xrightarrow{i\to\infty}0,\\
& Y(v_{i},H^{i},p_{i}, Q_{\theta})> C(L,m)\theta^{\alpha}\epsilon_{i}.
\end{align*}
Set
\begin{align*}
&{u}_{i}=\frac{v_{i}-(v_{i})_{1}}{\epsilon_{i}},\ \ \ {q}_{i}=\frac{p_{i}-(p_{i})_{1}(t)}{\epsilon_{i}}, \ \
{F}^{i}=\frac{H^{i}-(H^{i})_{1}}{\epsilon_{i}}, \ f^{i}=\frac{(v_{i})_{1}\otimes a_i}{\epsilon_{i}},\\
& g^{i}=\frac{M^{i}(H^i)_1^t+(H^i)_1(M^i)^t}{\epsilon_{i}}, \ \ \ \ \ \ \ \ h^{i}_j=\frac{(v_{i})_1\otimes M_j^i}{\epsilon_{i}},
\ \  l^i_j=\frac{a_i\otimes (H_j^i)_1}{\epsilon_i} (j=1,2,3).
\end{align*}
Then it is straightforward to verify that (${u}_{i}, {q}_{i}, F^{i}$) solves the system:
\begin{align}\label{E2.5}
\left\{
\begin{array}{llll}
\partial_{t}{u}_{i}-\Delta {u}_{i}+\nabla q_{i}=-\sigma_i\big[u_i\nabla a_i+(a_i+(v_i)_1+\epsilon_i u_i)\nabla u_i
+{\rm{div}}(f^i)\big]\\
\qquad\qquad\ \ \ \ \ \ +\sigma_i\nabla\cdot\big[\epsilon_i F^i(F^i)^t+(M^i+(H^i)_1)(F^i)^t+F^i(M^i+(H^i)_1)^t+g^i\big],\\
\partial_{t}F^{i}-\Delta F^{i}=-\sigma_i\big[(a_i+(v_i)_1+\epsilon_{i}u_i)\nabla F^i+u_i\nabla M^i+\sum_{j=1}^3\nabla_j(h_j^i)\big]\\
\qquad\qquad\quad\ \ +\sigma_i\big[\nabla u_i((H^i)_1+\epsilon_i F^i)+\nabla a_i F^i+\nabla u_i M^i +\sum_{j=1}^3\nabla_j(l_j^i)\big],\\
\mbox{div}u_{i}=\mbox{div}a_i=0, \ \mbox{div} F^i=\mbox{div} M^i=0.
\end{array}
\right.
\end{align}
Moreover, by direct calculations, we have that
\begin{align}\label{E2.5.1}
&\Big\{\aint_{Q}|u_{i}|^{3}\Big\}^{\frac{1}{3}}+\Big\{\aint_{Q}|F^{i}|^{3}\Big\}^{\frac{1}{3}}
+\Big\{\aint_{Q}|q_{i}|^{\frac{3}{2}}\Big\}^{\frac{2}{3}}+\Big\{\aint_{Q}|f^{i}|^{m}\Big\}^{\frac{1}{m}}\nonumber\\
&+\Big\{\aint_{Q}|g^{i}|^{m}\Big\}^{\frac{1}{m}}
+\sum_{j=1}^3\Big\{\aint_{Q}|h_j^{i}|^{m}\Big\}^{\frac{1}{m}}
+\sum_{j=1}^3\Big\{\aint_{Q}|l_j^{i}|^{m}\Big\}^{\frac{1}{m}}\le 1,
\end{align}
and
\begin{equation}\label{E2.5.2}
Y({u}_{i}, {F}^{i},q_{i}, Q_{\theta})> C(L,m)\theta^{\alpha}.
\end{equation}
Since ($v_i, p_i, H^i$) is a suitable weak solution of \eqref{E2.1}, it follows that
$(u_i, q_i, F^i)$ is also a suitable weak solution of \eqref{E2.5}. Hence 
the following local energy inequality holds:
\begin{align}\label{E2.6}
&\frac{d}{dt}\int_{\B}\phi\Big(|u_{i}|^{2}+|F^{i}|^{2}\Big)
+2\int_{\B}\phi\Big(|\nabla u_i|^{2}+|\nabla F^{i}|^{2}\Big)\leq\nonumber\\
\nonumber&\int_{\B}\bigg\{(|u_{i}|^{2}+|F^{i}|^{2})\big[(\phi_{t}+\Delta\phi)
+\sigma_i (a_i+(v_i)_1+\epsilon_i u_{i})\nabla\phi\big]
-(2\sigma_i\epsilon_{i}F^i(F^i)^t-q_i){u}_{i}\nabla \phi\\ 
&
+2\sigma_i\big[u_{i}a_{i}+(M^{i}+(H^i)_1)(F^i)^t
+F^i(M^i+(H^i)_1)^t+f^i-g^i\big]\cdot\nabla (u_{i}\phi)\\
\nonumber&+2\sigma_i\big[a_i F^i+u_i((H^i)_1+\epsilon_i F^i)+u_i M^i-M^i u_i\big]\cdot\nabla (F^i\phi)
+2\sigma_i\sum_{j=1}^3(h^i_j-l^i_j)\nabla_j(F^i\phi)
\bigg\},
\end{align}
for any $\phi\in C_{0}^{\infty}(\B_{1}(0)\times (-1,t])$, with $\phi\geq 0$.

We define
\begin{align*}
\displaystyle E_{i}(r)={\rm{esssup}}_{-r^{2}<t\leq0}\frac{1}{2}\int_{\B_{r}}\Big(|u_{i}|^2+|{F}^{i}|^{2}\Big)
+\int_{-r^{2}}^{0}\int_{\B_{r}}\Big(|\nabla {u}_{i}|^{2}+|\nabla {F}^{i}|^{2}\Big), \ 0\le r\le 1.
\end{align*}
By the known interpolation inequalities we have
\begin{equation}\label{E2.6.1}
\big\|{u}_{i}\big\|_{L^{\frac{10}{3}}(Q_{r})}+\big\|{F}^{i}\big\|_{L^{\frac{10}{3}}(Q_{r})}\leq CE^{\frac{1}{2}}_{i}(r).
\end{equation}
For any $\frac{1}{2}<r_{1}<r_{2}\leq 1$, by choosing suitable test functions $\phi\in C^\infty_0(Q_{r_2})$ in
the local energy inequality \eqref{E2.6}, applying the uniform bounds \eqref{E2.5.1} and H\"{o}lder's inequalities,
we have  that
\begin{align}\label{E2.7.1}
E_i(r_1)&\leq\nonumber\frac{C}{(r_{2}-r_{1})^{2}}
+\frac{C}{r_2-r_1}\int_{Q_{r_2}}\Big\{(|u_i|^2+|F^i|^2)(L+|a_i|+|M^i|+\epsilon_i|u_i|)+\nonumber\\
&\qquad\qquad\qquad\qquad\qquad\ \ \ \ \ \ (|f^i|+|g^i|+|q_i|)|u_i|+(|h^i|+|l^i|)|F^i|\Big\}\nonumber\\
&\quad+C\int_{Q_{r_2}}\Big\{\big[|a_i||u_i|+(L+|M^i|)|F^i|+|f^i|+|g^i|\big]|\nabla u_i|+\nonumber\\
&\qquad\qquad\ \ \ \ \ \ \big[(|a_i|+\epsilon_i|u_i|)|F^i|+(L+|M^i|)|u_i|+|h^i|+|l^i|\big]|\nabla F^i|\Big\}\nonumber\\
&\le \frac{C(L)}{(r_2-r_1)^2}+C\big\|(|f^i|+|g^i|+|h^i|+|l^i|)\big\|_{L^2(Q_{r_2})}E_{i}(r_2)^\frac12\nonumber\\
&\quad+C\big\|(|a_i|+|M^i)|\big\|_{L^5(Q_{r_2})}\big\|(|u_i|+|F^i|)\big\|_{L^{\frac{10}3}(Q_{r_2})}
\big\|(|\nabla u_i|+|\nabla F^i|)\big\|_{L^2(Q_{r_2})}\nonumber\\
&\le  \frac{C(L)}{(r_2-r_1)^2}+\left[C\big(\|a_{i}\|_{L^{m}(Q)}+\|M^i\|_{L^{m}(Q)}\big)+\frac{1}{8}\right]E_{i}(r_{2}).
\end{align}
If we choose $\eta>0$ so that $C\eta<\frac{1}{16}$, then a standard iteration argument shows that 
\begin{equation}\label{E2.7.2}
E_{i}(\frac{3}{4})\leq C(L,m,\eta).
\end{equation}
Indeed,  it follows from \eqref{E2.7.1} that there exists $0<\theta<\frac14$ such that
$$
E_{i}(r_{1})\leq \theta E_{i}(r_{2})+\frac{C(L)}{(r_{2}-r_{1})^{2}}, \ \forall\  \frac{1}{2}<r_{1}<r_{2}\leq 1.
$$
By taking $r_{2}=1, r_{1}=\frac{3}{4}$ and $\rho_{k}=r_{2}-2^{-k}(r_{2}-r_{1})$ for $k\ge 0$, we have 
by iterations that
$$
E_{i}(\rho_{0})\leq \theta^{k}E_{i}(\rho_{k})+\frac{C(L,m, \eta)}{1-4\theta},
$$
this implies (\ref{E2.7.2}) by sending $k\to\infty$.
Observe that by \eqref{E2.5}, we have
that 
\begin{equation}\label{E2.7.3}
\sup_{i\ge 1} \big\|(\partial_{t} u_{i}, \partial_t F^i)\big\|_{L^{\frac{3}{2}}\big((-\frac{9}{16},0), W^{-1,3}(\B_{\frac{3}{4}})\big)}
\le C.
\end{equation}
Thus, by Aubin-Lions' compactness lemma, we may assume that after passing to a subsequence, 
\begin{align*}
\left\{
\begin{array}{llll}
\sigma_i\rightarrow\sigma\in [0,1];\ \ u_{i}\to u,\  F^{i}\to F \  \mbox{in}\ \ L^{p}(Q_{\frac{3}{4}}); 
\ \ q_{i}\rightharpoonup q \ \  \mbox{in}\ L^{\frac{3}{2}}(Q_{\frac{3}{4}}),\\
(f^{i}, g^i)\rightharpoonup (f,g),\ \ (h^i_j, l^i_j)\rightharpoonup (h_j, l_j), j=1, 2,3,  \ \mbox{in}\ L^{m}(Q_{\frac{3}{4}},\R^{3\times 3}),\\
(v_{i})_{1}\to \lambda\in\R^3,\ (H^{i})_{1}\to \mu\in \R^{3\times 3}, \ \ (a_i, M^i)\rightharpoonup (a, M)
\ \mbox{in}\ L^{m}(Q_{\frac{3}{4}}).
\end{array}
\right.
\end{align*}
for $1\leq p<\frac{10}{3}$.  Passing to the limit in \eqref{E2.5}, we see 
that $(u,p, F)$ satisfies the generalized linear Stokes equations in $Q_\frac34$:
\begin{align}\label{E2.8.1}
\left\{
\begin{array}{llll}
\partial_{t}{u}-\Delta {u}+\nabla q=-\sigma\big[u\nabla a+(a+\lambda)\nabla u
+{\rm{div}} f\big]\\
\qquad\qquad\ \ \ \ \ \ \ \ \ \ \ +\sigma\nabla\cdot\big[(M+\mu)F^t+F(M+\mu)^t+g\big],\\
\partial_{t}F-\Delta F=-\sigma\big[(a+\lambda)\nabla F+u\nabla M+\nabla_j h_j\big]\\
\qquad\qquad\quad +\sigma\big[\nabla u\mu+\nabla a F+\nabla u M +\nabla_j l_j\big],\\
\mbox{div}u={\rm{div}}a=0, \ \mbox{div} F=\mbox{div} M=0.
\end{array}
\right.
\end{align}
By Lemma \ref{L2.1} below, we know that there exists $0<\beta=\beta(m)<1$ such that
$$
|u(x_{1},t_{1})-u(x_{2},t_{2})|+|F(x_1, t_1)-F(x_2,t_2)|
\leq C\big(|x_{1}-x_{2}|+|t_{1}-t_{2}|^{\frac{1}{2}}\big)^{\beta}, $$
for all $(x_1, t_1), \ (x_2, t_2)\in Q_{\frac34}$.

Since $(u_{i}, F^{i}) \to (u, F)\  \mbox{in}\ L^{3}(Q_{\frac{3}{4}})$, it follows  that
\begin{equation}\label{E2.8.1.1}
\Big(\aint_{Q_{\theta}}|u_{i}-({u}_{i})_{\theta}|^{3}\Big)^{\frac{1}{3}}+
\Big(\aint_{Q_{\theta}}|F^{i}-(F^{i})_{\theta}|^{3}\Big)^{\frac{1}{3}}\leq C(L,m,\eta)\theta^{\beta},
\end{equation}
for $i$ sufficiently large.

Taking divergence of \eqref{E2.5}$_1$ and using div$u_i=0$, we see that $q_i$ satisfies
\begin{eqnarray}\label{E2.8.2}
&\Delta q_i=-\sigma_i{\rm{div}}^2\big[u_i\otimes a_i+(a_i+(v_i)_1+\epsilon_i u_i)\otimes u_i
+f^i\big]\\
&\qquad\ \ \ \ \ \ \ \ \qquad\ \ \ \ \ \ \ \ \ +\sigma_i{\rm{div}}^2\big[\epsilon_i F^i(F^i)^t+(M^i+(H^i)_1)(F^i)^t
+F^i(M^i+(H^i)_1)^t+g^i\big], \nonumber
\end{eqnarray}
in $Q_\frac34$.

Write $q_{i}=q_{i}^{1}+q_{i}^{2}$, where $q_i^1$ is defined by
\begin{align*}
&q_{i}^{1}=\sigma_i(-\Delta)^{-1}\mbox{div}^2\Bigg\{
\big[u_i\otimes a_i+(a_i+(v_i)_1+\epsilon_i u_i)\otimes u_i
+f^i\big]\ \chi_{\B_\frac34}\\
&\qquad-\big[\epsilon_i F^i(F^i)^t+(M^i+(H^i)_1)(F^i)^t
+F^i(M^i+(H^i)_1)^t+g^i\big]\ \chi_{\B_\frac34}\Bigg\}\ \ {\rm{in}}\ \ Q_\frac34,
\end{align*}
where $\chi_{\B_\frac34}$ is the characteristic function of $\B_\frac34$.
Then it is easy to check that
$$\Delta q_{i}^{2}=0\ \ {\rm{in}}\ \ Q_{\frac34}.$$
Since $(u_i, {F}^{i})\to (u, F)$ in $L^{3}(Q_{\frac34})$, we have that
$$
\big\|q_i^1- q^1\big\|_{L^{\frac32}(Q_{\frac34})}\xrightarrow{i\rightarrow\infty} 0,
$$
where
$$
q^{1}=\sigma(-\Delta)^{-1}\mbox{div}^2\Big\{
\big[u\otimes a+(a+\lambda)\otimes u+f
-(M+\mu)F^t-F(M+\mu)^t-g\big] \chi_{\B_\frac34}\Big\}.
$$
It follows from Lemma \ref{L2.1} that $u$ and $F$ are bounded in $Q_{\frac12}$. Hence
by Calderon-Zygmund's $L^m$-estimate we have that
$$
\big\|q^1\big\|_{L^{m}(Q_{\frac12})}\leq C\left(\|a\|_{L^{m}(Q_{\frac12})}+\|M\|_{L^{m}(Q_{\frac12})}+L
+\|f\|_{L^m(Q_\frac12)}+\|g\|_{L^m(Q_\frac12)}\right),
$$
which, by H\"older's inequality, yields 
$$
\theta\Big(\aint_{Q_{\theta}}|q^{1}|^{\frac{3}{2}}\Big)^{\frac{2}{3}}
\leq C\theta\Big(\aint_{Q_{\theta}}|q^{1}|^{m}\Big)^{\frac{1}{m}}\leq C(L,m,\eta)\theta^{1-\frac{5}{m}}.
$$
Therefore, for $i$ sufficiently large, we have
$$
\theta\Big(\aint_{Q_{\theta}}|q_i^{1}|^{\frac{3}{2}}\Big)^{\frac{2}{3}}\leq C(L,m,\eta)\theta^{1-\frac{5}{m}}.
$$
Since $q_i^2(t)$ is harmonic in $\B_\frac34$ for all $-\frac9{16}\le t\le 0$, the standard estimate
implies that
$$\theta\Big(\aint_{Q_{\theta}}|q_i^{2}-(q_i^2)_\theta(t)|^{\frac{3}{2}}\Big)^{\frac{2}{3}}
\le C\theta^\frac23\Big(\aint_{Q_\frac34}|q^2_i|^\frac32\Big)^\frac23\le C\theta^\frac23.
$$
Putting these estimates together, we obtain
\begin{align}\label{E2.8.3}
\theta\left(\aint_{Q_{\theta}}|q_i-(q_i)_{\theta}(t)|^{\frac{3}{2}}\right)^{\frac{2}{3}}&\leq
\theta\left(\aint_{Q_{\theta}}|q_{i}^{1}|^{\frac{3}{2}}\right)^{\frac{2}{3}}+
\theta\left(\aint_{Q_{\theta}}|q_i^{2}-(q_i^{2})_{\theta}(t)|^{\frac{3}{2}}\right)^{\frac{2}{3}}\nonumber\\
&\leq C(L,m,\eta)\theta^{\min\{\frac23, 1-\frac5{m}\}},
\end{align}
for $i$ sufficiently large. This, together with \eqref{E2.8.1.1}, shows that
$$
Y(u_{i}, F^{i},q_{i}, Q_{\theta})\leq  C(L,m,\eta)\theta^{\min\{\beta,\frac23, 1-\frac5m\}},
$$
for $i$ sufficiently large, which contradicts to \eqref {E2.5.2} if $\alpha\in (0,\frac12)$ is chosen to be sufficiently small. 
This completes the proof.
\end{proof}
Now we need to prove the uniform regularity of the following non-homogeneous generalized Stokes systems
\eqref{E2.8.1} for $0\le\sigma\le 1$.
\begin{Lemma}\label{L2.1} For any $m>5$, let $a\in L^m(Q_1,\R^3)$, with ${\rm{div}}a=0$,
and $M, f, g, h_j, l_j\in L^m(Q_1,\R^{3\times 3})$ for $j=1,2, 3$, with $\mbox{div} M=0$, $\lambda\in\R^3, \mu\in \R^{3\times 3}$, be such that
\begin{equation}\label{E2.2.1}
\big\||a|+|M|+|f|+|g|+\sum_{j=1}^3(|h_j|+| l_j)|\big\|_{L^m(Q_1)}+|\lambda|+|\mu|\le L,
\end{equation}
for some positive $L>0$.  
Let $(u, F)\in L_{t}^{\infty}L_{x}^{2}(Q_1)\cap L_{t}^{2}{\bf V}(Q_1)$ and $p\in L^{\frac{3}{2}}(Q_1)$ satisfy
\begin{equation}\label{E2.2.2}
\big(\int_{Q_1}|u|^{3}\big)^\frac13+\big(\int_{Q_1}|F|^{3}\big)^{\frac{1}{3}}
+\big(\int_{Q_1}|p|^{\frac{3}{2}}\big)^{\frac{2}{3}}\leq L,
\end{equation}
and solve the system \eqref{E2.8.1}, for some $\sigma\in [0,1]$, on $Q_1$ in the sense of distributions. 
Then $(u,F)$ is H\"{o}lder continuous in $Q_{\frac{1}{2}}$ with an
exponent $0<\beta=\beta(m, L)<1$, and
\begin{equation}\label{E2.2.3}
\big\|u\big\|_{C_{par}^{\beta}(Q_{\frac{1}{2}})}+\big \|F\big\|_{C_{par}^{\beta}(Q_{\frac{1}{2}})} \leq C(m,L).
\end{equation}
\end{Lemma}

\begin{proof} The proof, similar to \cite{JS} Lemma 2.2, is based on some standard bootstrapping arguments.
We will sketch it here. Set
$$
N:=u\otimes a+(a+\lambda)\otimes u+f-(M+\mu)F^t-F(M+\mu)^t-g.
$$
Then $N\in L^{\frac{mq}{m+q}}(Q_1)$ and 
$$\|N\|_{L^{\frac{mq}{m+q}}(Q_1)}\le C(L)\big\|(|u|+|F|)\big\|_{L^q(Q_1)}.$$
Taking divergence of \eqref{E2.8.1}$_1$ gives
$$-\Delta p=\sigma{\rm{div}}^2N\ \  {\rm{in}}\ Q_1.$$
For $R\in [\frac34,1]$, set
$$p_1=\sigma(-\Delta)^{-1}{\rm{div}}^2(N\chi_{\B_R}),$$
and $p_2=p-p_1$. Then 
$$\Delta p_2=0 \ \ {\rm{in}}\ \ \B_R.$$
Assume $(u, F)\in L^q(Q_1)$ for $q\ge 3$.  Direct calculations yield
$$\big\|p_1\big\|_{L^{\frac{mq}{m+q}}(Q_R)}\le C\big\|N\big\|_{L^{\frac{mq}{m+q}}(Q_R)}
\le C(L)\big\|(|u|+|F|)\big\|_{L^q(Q_1)},$$
and 
$$\big\|p_2\big\|_{L^\frac32_tC^2_x(Q_{(1-\delta)R})}\le C(\delta, L),$$
for any small $0<\delta<1$.

Let $\eta$ be a smooth cutoff function such that
$\eta=1$ in $Q_{(1-2\delta)R}$ and $\eta=0$ outside $Q_{(1-\delta)R}$.
Decompose $u$ by letting $u=u_{1}+u_{2}+u_{3}$, where 
\begin{align*}
&u_{1}(\cdot,t)=\int_{-\infty}^{t}e^{(t-s)\Delta}\big[-\nabla(p_{1}\eta)-\sigma\mbox{div}(N\eta)\big](\cdot, s)\,ds,\\
&u_{2}(\cdot,t)=-\int_{-\infty}^{t}e^{(t-s)\Delta}\nabla (p_{2}\eta)(\cdot,s)\,ds.
\end{align*}
By the standard estimates on heat kernel (see \cite{K}), we have that $u_1\in L^{\gamma}(Q_{(1-\delta)R})$
for any $\gamma\ge 1$ satisfying $\frac{1}{\gamma}>\frac1{q}+\frac1{m}-\frac15$, and
$$\big\|u_1\big\|_{L^\gamma(Q_{(1-\delta)R})}\le C\big\|(|p_1|+|N|)\big\|_{L^{\frac{mq}{m+q}}(Q_R)},$$
and 
$$\big\|u_2\big\|_{L^\infty(Q_{(1-\delta)R})}\le C\big\|p_2\big\|_{L^\frac32_tC^2_x(Q_{(1-\delta)R})}\le C(\delta, L).$$ 
Since $u_3$ solves the heat equation in $Q_{(1-2\delta)R}$, it follows that
$$\big\|u_3\big\|_{L^\infty(Q_{(1-3\delta)R})}\le C\Big(\big\|u\big\|_{L^q(Q_R)}
+\big\|u_1\big\|_{L^q(Q_R)}+\big\|u_2\big\|_{L^q(Q_R)}\Big)
\le C(\delta, L)\Big(1+\big\|(|u|+|F|)\big\|_{L^q(Q_R)}\Big).$$
Putting these estimates together, we obtain
$$\big\|u\big\|_{L^\gamma(Q_{(1-3\delta)R})}\le C(L,\delta)\Big(1+\big\|(|u|+|F|)\big\|_{L^q(Q_R)}\Big),$$
for $\gamma>1$ as long as $\frac{1}{\gamma}>\frac1{q}+\frac1{m}-\frac15$.

To estimate $F$. Decompose $F$ by letting $F=F_1+F_2$, where $F_1$ is defined
by
\begin{align*}
&F_1^{ij}(\cdot, t)=
-\sigma\int_{-\infty}^t e^{(t-s)\Delta}\nabla_k\Big[\big((a^k+\lambda^k) F^{ij}+u^kM^{ij}+h_k^{ij}\big)\eta\\
&\qquad\qquad\qquad\qquad -\big(u^i\mu^{kj}+a^iF^{kj}+u^iM^{kj}+l_k^{ij}\big)\eta\Big](\cdot, s)\,ds,
\ 1\le i, j\le 3.
\end{align*}
Thus $F_2$ satisfies the heat equation in $Q_{(1-2\delta)R}$. 

Similar to the estimate of $u_1$, we have that
$F_1\in L^{\gamma}(Q_{(1-\delta)R})$ for any $\gamma>1$ with $\frac{1}{\gamma}>\frac1{q}+\frac1{m}-\frac15$,
and
$$\big\|F_1\big\|_{L^\gamma(Q_{(1-\delta)R})}\le C(L)\Big(1+\big\|(|u|+|F|)\big\|_{L^{\frac{mq}{m+q}}(Q_R)}\Big).$$
Similar to the estimate of $u_3$, we have that $F^2\in L^\infty(Q_{(1-3\delta)R})$ and
$$\big\|F_2\big\|_{L^\infty(Q_{(1-3\delta)R})}\le C(L,\delta)\Big(1+\big\||u|+|F|\big\|_{L^q(Q_R)}\Big).$$
Combining these two estimates yields
$$\big\|F\big\|_{L^\gamma(Q_{(1-3\delta)R})}\le C(L,\delta)\Big(1+\big\|(|u|+|F|)\big\|_{L^q(Q_R)}\Big),$$
for $\gamma>1$ as long as $\frac{1}{\gamma}>\frac1{q}+\frac1{m}-\frac15$.

By repeating these arguments for finitely many times, it follows that 
$(u, F)\in L^{\gamma}(Q_{\frac34})$ for any finite $\gamma>3$. The interior H\"older continuity of
$(u,F)$, along with uniform estimate (\ref{E2.2.3}),  then follows from the standard estimates for the Stokes equations and heat equations. The completes the proof.
\end{proof}


By translation and dilation and iterations,  Lemma \ref{L2.1} implies

\begin{Lemma}\label{L2.3} 
Let $(u,p,F)$, $a, M$, $\eta, L$, $\epsilon(\theta, L,m)$, $C(L, m)$,  $\alpha=\alpha(m)$ be the same as in Lemma \ref{L2.1}.
Let $\beta=\frac{\alpha}{2}$ and $\theta_0\in (0,\frac12)$ be such that 
$C(L,m)\theta^{\alpha-\beta}\leq1$ for $\theta\le \theta_0$. Then there is $\epsilon_{1}=\epsilon_1(\theta_0, L, m)>0
$ 
such that if, for $Q_r(z_0)\subset Q_1$, 
\begin{align}\label{Cond2.5}
\left\{
\begin{array}{llll}
\displaystyle r|(u)_{z_0,r}|+r|(F)_{z_0,r}|<\frac{L}{2},\\
\displaystyle rY(u,p,F, Q_r(z_0))+rL\Big(\aint_{Q_r(z_0)}|a|^{m} +|M|^{m}\Big)^\frac{1}{m}\leq \epsilon_{1},
\end{array}
\right.
\end{align}
then for any $k\ge 0$ it holds that
\begin{align}\label{Conc2.6}
\left\{
\begin{array}{llll}
\displaystyle  r|(u)_{z_0, \theta^{k} r}|+r|(F)_{z_0, \theta^{k}r}|\leq L,\\
\displaystyle rY\big(u,p,F,Q_{\theta^{k}r}(z_0)\big)+r\theta^{k}\Big\{|(u)_{Q_{\theta^{k}}(z_0)}|+
|(F)_{Q_{\theta^{k}r}(z_0)}|\Big\}\Big\{\aint_{Q_{\theta^{k}r}(z_0)}|a|^{m} +|M|^{m}\Big\}^\frac{1}{m}\leq \epsilon_{1},\\
\displaystyle Y\big(u,p,F,Q_{\theta^{k+1}r}(z_0)\big)
\leq \theta^{\beta}\Big\{Y\big(u,p,F,Q_{\theta^{k}r}(z_0)\big)\\
+r\theta^{k}\Big(|(u)_{Q_{\theta^{k} r}(z_0)}|+
|(F)_{Q_{\theta^{k}r}(z_0)}|\Big)\Big(\aint_{Q_{\theta^{k}r}(z_0)}|a|^{m} +|M|^{m}\Big)^\frac{1}{m}\Big\}.
\end{array}
\right.
\end{align}
\end{Lemma}
\begin{proof} By translation and dilation, it suffices to show this Lemma for $z_0=0$ and $r=1$.
It is easy to see that the conclusion for $k=0$ follows from Lemma \ref{L2.2}. Observe that
by the triangle inequality we have
\begin{equation}\label{E2.8}
\big|(u)_{\theta^{k+1}}\big|+\big|(F)_{\theta^{k+1}}\big|\leq
\theta^{-\frac53}\sum\limits_{j=0}^{k} Y(u,p,F,Q_{\theta^{j}})+\big|(u)_{Q_1}\big|+\big|(F)_{Q_1}\big|.
\end{equation}
By induction, we may assume that (\ref{Conc2.6}) holds for $k=1,\cdot\cdot\cdot, k_0$.
Then we see that
\begin{align*}
 &Y(u,p,F,Q_{\theta^{k+1}})\leq \theta^{\beta}\Big\{Y(u,p,F,Q_{\theta^{k}})+
 \theta^{k}L\big(\aint_{Q_{\theta^k}}|a|^{m} +|M|^{m}\big)^\frac{1}{m}\Big\}\\&
 \leq\theta^{\beta}\Big\{Y(u,p,F,Q_{\theta^{k}})+
 \theta^{k(1-\frac{5}{m})}L\big(\aint_{Q_1}|a|^{m} +|M|^{m}\big)^\frac{1}{m}\Big\}\\
 &\leq \theta^{\beta} Y(u,p,F,Q_{\theta^k})+\theta^{(k+1)\beta_{1}}\epsilon_{1},
\end{align*}
holds for $k=1,\cdots, k_0$, where $\beta_{1}=\min\{\beta,1-\frac{5}{m}\}$. 
Iterations of the above inequalities yield that 
\begin{equation}\label{E2.9}
Y(u,p,F,Q_{\theta^{k+1}})\leq \theta^{(k+1)\beta}Y(u,p,F,Q_1)+(k+1)\theta^{(k+1)\beta_{1}}\epsilon_{1}
\end{equation}
holds for $k=1,\cdots, k_0$.

We now show that (\ref{Conc2.6}) holds for $k=k_0+1$. 
Combing \eqref{E2.8} with \eqref{E2.9}, we have
\begin{align*}
\big|(u)_{\theta^{k_0+1}}|+|(F)_{\theta^{k_0+1}}|
&\leq \theta^{-\frac{5}{3}}\sum\limits_{j=0}^{k_0}\big(\theta^{j\beta}\epsilon_{1}+
j\theta^{j\beta_{1}}\epsilon_{1}\big)+\frac{L}{2}\\&
\leq\theta^{-\frac{5}{3}}\epsilon_1\Big[\frac{1}{1-\theta^{\beta}}+C(\beta_1,\theta)\Big]+\frac{L}{2}\\&\leq L,
\end{align*}
provided $\epsilon_1(\theta, L,m)$ is chosen to be sufficiently small.

While for \eqref {Conc2.6}$_2$, we have that
\begin{align*}
&Y(u,p,F,Q_{\theta^{k_0+1}})+\theta^{k_0+1}\Big(|(u)_{\theta^{k_0+1}}|+
|(F)_{\theta^{k_0+1}}|\Big)\Big(\aint_{Q_{\theta^{k_0+1}}}|a|^{m} +|M|^{m}\Big)^\frac{1}{m}\\
&\leq \theta^{\beta}\epsilon_1+L\theta^{k_0+1}\Big(\aint_{Q_{\theta^{k_0+1}}}|a|^{m} +|M|^{m}\Big)^\frac{1}{m}\\
&\leq \theta^{\beta}\epsilon_1+\theta^{(1-\frac{5}{m})(k_0+1)}L\Big(\aint_{Q_1}|a|^{m} +|M|^{m}\Big)^\frac{1}{m}\\
&\leq \theta^{\beta}\epsilon_1+\theta^{(1-\frac{5}{m})(k_0+1)}\epsilon_1<\epsilon_1.
\end{align*}
Now we rescale $(u, p, F)$ by letting
$$\Big(v, q, H\Big)(x,t)=\Big(\theta^{k_0+1}u, \theta^{2(k_0+1)}p, \theta^{k_0+1}F\Big)
\big(\theta^{k_0+1}x, \theta^{2(k_0+1)}t\big),$$
and
$$
\Big(b, N \Big)(x,t)=\theta^{k_0+1}\Big(a, M\Big)
\big(\theta^{k_0+1}x, \theta^{2(k_0+1)}t\big),
$$
for $(x,t)\in Q_1$. Then $(v, q, H)$ is  a suitable weak solution of (\ref{E2.1}) on $Q_1$ for $\sigma\in [0,1]$, with
$a, M$ replaced by $b, N$. Moreover,
\begin{align*}
&Y(v,q,H, Q_1)+\Big\{|(v)_1|+|(H)_1|\Big\}\Big(\aint_{Q_1}|b|^m+|N|^m\Big)^\frac1m\\
&=\theta^{k_0+1}\Big(Y(u,p,F, Q_{\theta^{k_0+1}})+\theta^{k_0+1}
\Big\{|(u)_{\theta^{k_0+1}}|+|(H)_{\theta^{k_0+1}}|\Big\}\Big(\aint_{Q_{\theta^{k_0+1}}}|b|^m+|N|^m\Big)^\frac1m \Big)
\le\epsilon_1,
\end{align*}
$$\Big\{|(v)_1|+|(H)_1|\Big\}=\theta^{k_0+1}\Big\{|(u)_{\theta^{k_0+1}}|+|(F)_{\theta^{k_0+1}}|\Big\}\le L,$$
and
$$\big\|b\big\|_{L^m(Q_1)}+\big\|N\big\|_{L^m(Q_1)}=\theta^{(k_0+1)(1-\frac{5}m)}
\Big(\big\|a\big\|_{L^m(Q_{\theta^{k_0+1}})}+\big\|N\big\|_{L^m(Q_{\theta^{k_0+1}})}\Big)\le \delta.$$
Thus by Lemma \ref{L2.2} we get
$$Y(v,q, H,Q_{\theta})\le\theta^\beta\Big[
Y(v, q,H, Q_1)+\Big\{|(v)_{1}|+|(H)_1|\Big\}\Big\{\aint_{Q_1}|b|^m+|N|^m\Big\}^\frac1m\Big],$$
which, after rescaling, gives (\ref{E2.9})$_3$ for $k=k_0+1$.
The proof is completed.
\end{proof}


\noindent {\it Proof of Theorem \ref{Th2.1}.} By choosing $\epsilon_{0}$ sufficiently small, we can apply Lemma \ref{L2.3}
 in $Q_{\frac12}(z_0)$ for any $z_{0}\in Q_{\frac12}$ to conclude that for some $\beta\in (0,1)$ depending
 on $m, L$ such that 
\begin{align*}
\Big(\aint_{Q_\rho(z_0)}|u-(u)_{z_0,\rho}|^{3}\Big)^{\frac{1}{3}}
+\Big(\aint_{Q_\rho(z_0)}|F-(F)_{z_0, \rho}|^{3}\Big)^{\frac{1}{3}}
\leq Y(u,p,F,Q_\rho(z))\leq C(\theta, \beta)\rho^{\beta},
\end{align*}
 for any $\rho\leq\frac12$. By Campanato's Lemma, we conclude that
 $u,F$ are H\"{o}lder continuous in $Q_{\frac12}$ with the desired estimates.
 \qed

\setcounter{section}{2}
\setcounter{equation}{0}
\section{Existence of self-similar solutions}
\subsection{Local in space near initial time estimate of local weak solutions}

In this subsection, we will apply the ``$\epsilon$-regularity" Theorem  \ref{Th2.1}
to obtain the local in space near initial time estimate of local {\em Leray
weak solutions} of (1.3)$_\sigma$. To this end, we first introduce the definition of local {\em Leray
weak solutions} of the system (1.3)$_\sigma$, analogous to (\ref{E2.2+}) for the system 
\eqref{E2.1} (see \cite{JS,LR} on the Navier-Stokes
equation).
\begin{Definition}
For any $u_0\in L^2_{\rm{loc}}(\R^3,\R^3)$ with $\mbox{div} u_{0}=0$ and 
$F_0\in L^2_{\rm{loc}}(\R^3,\R^{3\times 3})$ with $\mbox{div} F =0$, a pair of functions
$u\in L^2_{\rm{loc}}(\R^3\times [0,\infty),\R^3)$ and $F\in L^2_{\rm{loc}}(\R^3\times [0,\infty),\R^{3\times 3})$
is called a local-Leray solution to the system (1.3)$_\sigma$ for  $\sigma\in [0,1]$, with initial data $u_{0}, F_{0}$,  if 
\begin{itemize}
\item[(i)] For any $0<R<\infty$, it holds
$$\mbox{esssup}_{0\leq t<R^{2}, x_{0}\in\R^{3}}\frac{1}{2}\int_{\B_{R}(x_{0})}(|u|^{2}+|F|^{2})+
\sup_{x_{0}\in\R^{3}}\int_{\B_{R}(x_{0})\times [0,R^2]}(|\nabla u|^{2}+|\nabla F|^{2})<\infty,$$ 
and
$$
\lim_{|x_{0}|\to\infty}\int_{\B_{R}(x_{0})\times [0,R^2]}(|u|^{2}+|F|^{2})=0.
$$
\item [(ii)] There exists  a distribution $p$ in $\R^{3}\times (0,\infty)$ such that $(u,F,p)$ satisfies (1.3)$_\sigma$
for $\sigma\in [0,1]$  in the sense of distributions, and, for any compact set
$K\subseteq \R^{3}$,
$$\big\|u(\cdot,t)-u_{0}\big\|_{L^{2}(K)}+\big\|F(\cdot,t)-F_{0}\big\|_{L^{2}(K)}\xrightarrow{t\rightarrow 0}0.$$

\item [(iii)] $(u,F)$ is suitable in the sense that the following local energy inequality holds (see also \cite{C}): 
\begin{align}
\label{NE3.1}
 \begin{split}
 &\int_{\R^{3}}\phi\big(|u|^{2}+|F|^{2}\big)+2\int_{\R^3\times [0, \infty]}\phi\big(|\nabla u|^{2}+|\nabla F|^{2}\big)\\
&\leq\int_{\R^\times [0, \infty]}\Bigg\{(\phi_{t}+\Delta \phi)\big(|u|^{2}+|F|^{2}\big)-2\sigma u\otimes \nabla\phi\cdot FF^{t}\\
&\qquad\qquad\ \ \ +\big[\sigma(|u|^{2}+|F|^{2})+2p\big]u\cdot\nabla\phi\Bigg\},
\end{split}
\end{align}
for any $0\le \phi\in C^\infty(\R^{3}\times (0,\infty))$ with supp $\phi\Subset\R^3\times (0,\infty)$. 
\end{itemize}
\end{Definition}
The set of all local Leray weak solutions to (1.3)$_\sigma$ with initial data $(u_{0},F_{0})$ will be denoted
as $\mathcal{N} (u_{0}, F_{0})$.

 The proof of the local existence of local-Leray weak solutions to the system (1.3)$_\sigma$ is standard.
 However, since we cannot find it in the literature, we sketch a proof, that is based on \cite{LR0, LR},
 for the reader's convenience.

\begin{Theorem} 
For $u_0\in L^2_{\rm{loc}}(\R^3,\R^3)$ with $\mbox{div} u_{0}=0$,
$F_0\in L^2_{\rm{loc}}(\R^3, \R^{3\times 3})$ with $\mbox{div} F_{0}=0$, 
if there exists $0<R<+\infty$ such that
$$\sup_{x_{0}\in \R^{3}}\int_{\B_{R}(x_{0})}\big(|u_{0}|^{2}+|F_{0}|^{2}\big)<\infty, \ \
\lim_{x_{0}\to\infty}\int_{\B_{R}(x_{0})}\big(|u_{0}|^{2}+|F_{0}|^{2}\big)=0,$$ 
then there exists $T_*=T_*(R)>0$ and  a local Leray solution $(u,F)$ of the system (1.3)$_\sigma$, for any $\sigma\in [0,1]$,
on $\R^{3}\times [0,T_*)$.
\end{Theorem}
\begin{proof} For $f\in L^2_{\rm{loc}}(\R^3)$ and $R>0$, set
$$\|f\|_{L^2_R(\R^3)}:=\sup_{x\in\R^3}\Big(\int_{\mathbb B_R(x)}|f|^2\Big)^\frac12.$$
It is standard that there exist
$u_{0}^{k}\in C_{0}^{\infty}(\R^{3},\R^3)$ and $F_0^k\in C_0^\infty(\R^{3},\R^{3\times 3})$,
with $\mbox{div} (u_0^k)=0$ and $\mbox{div}(F_0^k)=0$, such that
$$
(u_{0}^{k}, F_0^k)\xrightarrow{k\rightarrow 0} (u_{0}, F_{0})
\ \  \mbox{in}\ \ L^2_{\rm{loc}}(\R^3),$$
and
$$\sup_k\|u_0^k\|_{L^2_R(\R^3)}\le 2\|u_0\|_{L^2_R(\R^3)},
\  \sup_k\|F_0^k\|_{L^2_R(\R^3)}\le 2\|F_0\|_{L^2_R(\R^3)}.
$$
By Leray's procedure \cite{Ler}, there exist $T_{k}>0$ and smooth solutions $(u^{k}, p^k, F^{k})$ to (1.3)$_\sigma$
on $\R^3\times [0,T_k)$,  under the initial conditions  $(u^k, F^k)\big|_{t=0}=(u_{0}^{k}, F_{0}^{k})$, such that
$$\begin{cases} (u^{k}, F^{k})\in L_{t}^{\infty}L_{x}^{2}\cap L_{t}^{2}H^{1}(\R^{3}\times [0,T_{k})),\\
\displaystyle\lim_{t\to0}\|u^{k}-u_{0}^{k}\|_{L^{2}(\R^{3})}=\lim_{t\to0}\|F^{k}-F_{0}^{k}\|_{L^{2}(\R^{3})}=0.
\end{cases}
$$
Employing the same argument as in  Lemma \ref{L3.1} below, we can conclude 
that there exist $T_*>0$ and $C_0>0$, independent of $k$ and $\sigma$, such that $T_{k}\ge T_*$, and 
\begin{equation}\label{NE3.2}
\sup_{0\le t\le T_*}\Big\{\|u^{k}(t)\|_{L^{2}_R(\R^{3})}+\|F^{k}(t)\|_{L^{2}_R(\R^{3})}\Big\}+
\ \sup_{x_{0}\in\R^{3}}\int_{0}^{T_*}\int_{\B_{R}(x_{0})}\Big(|\nabla u^{k}|^{2}+|\nabla F^{k}|^{2}\Big)\le C_0.
\end{equation}
and there exists $p^{k}(t)=p^{k}_{x_{0},R}(t)\in\R$ such that
 \begin{equation}\label{NE3.3}
 \sup_{x_{0}\in\R^{3}}\int_{0}^{T_*}\int_{\B_{R}(x_{0})}|p^{k}(x,s)-p^{k}(s)|^{\frac{3}{2}}\le C_0.
 \end{equation}
Thus we may employ an argument, similar to Lemma \ref{L2.1}, to conclude that after passing to a subsequence, 
$(u^{k}, F^{k}, p^{k})$ converges weakly in $L^2_{\rm{loc}}(\R^3\times [0,T_*))$
to a weak solution $(u, F, p)$ of the system  (1.3)$_\sigma$, under the initial condition $
(u,F)\big|_{t=0}=(u_{0}, F_{0})$, on $\R^3\times [0,T_*)$.

It remains to verify that  $(u, F, p)$  is a suitable weak solution of (1.3)$_\sigma$ on $\R^3\times [0,T_*)$, i.e., 
satisfying the local energy inequality \eqref{NE3.1}. In fact, since $(u^{k}, F^{k}, p^{k})$ is smooth in $\R^3\times [0,T_*)$,
we have
\begin{align}\label{NE3.4}
\begin{split}
&\partial_{t}\Big(|u^{k}|^{2}+|F^{k}|^{2}\Big)+2\Big(|\nabla u^{k}|^{2}+|\nabla F^{k}|^{2}\Big)\\
&=\Delta\Big(|u^{k}|^{2}+|F^{k}|^{2}\Big)
+\mbox{div}\Big[\sigma (|u^{k}|^{2}u^{k}+|F^{k}|^{2}u^{k})+2p^{k}u^{k}+2\sigma u^{k}\cdot F^{k}(F^{k})^{t}\Big].
\end{split}
\end{align}
Since $(u^k, F^k, p^k)$ satisfies (1.3)$_\sigma$ in $\R^3\times (0, T_*)$,
using the bounds \eqref{NE3.2} and \eqref{NE3.3} we can apply Aubin-Lions' compactness lemma to
to conclude that 
$$(u^k, F^k)\xrightarrow{k\rightarrow\infty} (u, F) \ \ {\rm{in}}\ \ L^3_{\rm{loc}}(\R^3\times [0, T_*)).$$
Hence  we have 
$$ |u^{k}|^{2}+|F^{k}|^{2}\xrightarrow{k\rightarrow\infty} |u|^{2}+|F|^{2},
$$
and
$$
\sigma(|u^{k}|^{2}u^{k}+|F^{k}|^{2}u^{k})+2p^{k}u^{k}+2\sigma u^{k}\cdot F^{k}(F^{k})^{t}
\xrightarrow{k\rightarrow\infty}\sigma( |u|^{2}u+|F|^{2}u)+2pu+2\sigma u\cdot FF^{t},$$
in $L^1_{\rm{loc}}(\R^3\times [0,T_*))$.

It follows from (\ref{NE3.2}) that 
$$
(|\nabla u^{k}|^{2}+|\nabla F^{k}|^{2})\,dxdt\rightharpoonup (|\nabla u|^{2}+|\nabla F|^{2})\,dxdt+\mu,
$$
as convergence of Radon measures as $k\rightarrow\infty$, where 
$\mu$ is a nonnegative Radon measure on $\R^3\times [0, T_*)$. 
Combining these convergences with \eqref{NE3.4}, we can show that the local energy inequality \eqref{NE3.1}
holds.

Finally we want to show that for any compact set $K\subseteq \R^{3}$,
$$
\lim_{t\to0}\Big(\|u(\cdot,t)-u_{0}\|_{L^{2}(K)}+\|F(\cdot,t)-F_{0}\|_{L^{2}(K)}\Big)=0.
 $$
Indeed, since for any $\varphi\in C_{0}^{\infty}(\R^{3})$, $t\to\int_{\R^{3}}(|u^{k}(x,t)|^{2}+|F^k(x,t)|^2)\varphi(x) $ is smooth on $[0,T^{*})$, we have
\begin{eqnarray}\label{NE3.5}
\int_{\R^{3}}(|u^{k}(x,t)|^{2}+|F^k(x,t)|^2)\varphi(x) &=&\int_{\R^{3}}(|u^{k}(x,0)|^{2}+|F^k(x,0)|^2)\varphi(x)\nonumber\\ 
&+&\int_{0}^{t}\int_{\R^{3}}\partial_{t}(|u^{k}(x,s)|^{2}+|F^k(x,s)|^2)\varphi(x).
\end{eqnarray}
By \eqref{NE3.4}, we have
\begin{align*}
 \begin{split}
\int_{0}^{t}\int_{\R^{3}}\partial_{t}(|u^{k}|^{2}+|F^k|^2)\varphi&=\int_{0}^{t}\int_{\R^{3}}\Bigg\{\Delta \phi\Big(|u^{k}|^{2}+|F^{k}|^{2}\Big)
-2\sigma u^{k}\otimes \nabla\phi\cdot F^{k}(F^{k})^{t}\\
&\qquad+\Big[\sigma(|u^{k}|^{2}+|F^{k}|^{2})+2p^{k}\Big]u^{k}\cdot\nabla\phi\Bigg\}:=A_{k}(\varphi,t).
\end{split}
\end{align*}
Since $(u^k,F^k)\rightarrow (u, F)$ in $L^3_{\rm{loc}}(\R^3\times [0,T_*))$, it follows that
$$ A_{k}(\varphi,t)\xrightarrow{k\rightarrow\infty}A(\varphi,t),$$ 
where 
$$ A(\varphi,t)=\int_{0}^{t}\int_{\R^{3}}\Bigg\{\Delta \phi\Big(|u|^{2}+|F|^{2}\Big)
-2\sigma u\otimes \nabla\phi\cdot FF^{t}
+\Big[\sigma(|u|^{2}+|F|^{2})+2p\Big]u\cdot\nabla\phi\Bigg\}.
$$
Sending $k\rightarrow\infty$ in \eqref{NE3.5} yields that for any $t>0$,
\begin{equation}\label{NE3.5.1}
\int_{\R^{3}}(|u(x,t)|^{2}+|F(x,t)|^2)\varphi(x) \le \int_{\R^{3}}(|u_0(x)|^{2}+|F_0(x)|^2)\varphi(x)
+A(\varphi,t).
\end{equation}
This, in particular, implies
$$\limsup_{t\rightarrow 0}\int_{\R^{3}}(|u(x,t)|^{2}+|F(x,t)|^2)\varphi(x)\le \int_{\R^{3}}(|u_0(x)|^{2}+|F_0(x)|^2)\varphi(x).$$
On the other hand, since 
$$(u(\cdot, t), F(\cdot, t))\rightharpoonup (u_0(\cdot), F_0(\cdot))
\ \mbox{in}\ L^2_{\rm{loc}}(\R^3),$$
it follows from the lower semicontinuity that
$$\liminf_{t\rightarrow 0}\int_{\R^{3}}(|u(x,t)|^{2}+|F(x,t)|^2)\varphi(x)\ge \int_{\R^{3}}(|u_0(x)|^{2}+|F_0(x)|^2)\varphi(x).$$ 
Thus we obtain
 $$
\lim_{t\to0}\int_{\R^{3}}(|u(x,t)|^{2}+|F(x,t)|^2)\varphi(x)=\int_{\R^{3}}(|u_0(x)|^{2}+|F_0(x)|^2)\varphi(x).$$
This completes the proof.
\end{proof}

We now give a few estimates for local Leray- weak solution of (1.3)$_\sigma$, for $\sigma\in [0,1]$.
\begin{Lemma}\label{L3.1}
There exist $0<C_{1}<1<C_{2}$ such that if $u_0\in L^2_{\rm{loc}}(\R^3,\R^3)$, with ${\rm{div}} u_0=0$,
$F_0\in L^2_{\rm{loc}}(\R^3,\R^{3\times 3})$, with ${\rm{div}} F_0=0$, satisfies
$$
 \alpha=\|u_0\|_{L^2_R(\R^3)}^2+ \|F_0\|_{L^2_R(\R^3)}^2<\infty$$
 for some $0<R<\infty$,  and $(u,F)$ is a local Leray weak solution of (1.3)$_\sigma$ for $\sigma\in [0,1]$,
 under the initial data $(u_{0},F_{0})$, on $\R^3\times [0,T_*)$ for some $0<T_*<\infty$. Then,
 for $\lambda=C_{1}\min(\alpha^{-2}R^{2},1)$, it holds
 \begin{equation}\label{NE3.6}
 \mbox{ess}\sup_{0\leq t<\lambda R^{2}}\sup_{x_{0}\in\R^{3}}\int_{\B_{R}(x_{0})}(|u|^{2}+|F|^{2})(x,t)+
\sup_{x_{0}\in\R^{3}}2\int_{0}^{\lambda R^{2}}\int_{\B_{R}(x_{0})}\Big(|\nabla u|^{2}+|\nabla F|^{2}\Big)\le C_{2}\alpha,
\end{equation}
 and, for some $p(t)=p_{x_{0},R}(t)\in\R$,
 \begin{equation}\label{NE3.7}
 \sup_{x_{0}\in\R^{3}}\int_{0}^{\lambda R^{2}}\int_{\B_{R}(x_{0})}|p(x,t)-p(t)|^{\frac{3}{2}}\le C_{2}\alpha^{\frac{3}{2}}R^{\frac{1}{2}}.
 \end{equation}
 \end{Lemma}

\begin{proof}\ It follows from the local energy inequality \eqref{NE3.1} that for any $x_0\in \R^3$, 
\begin{align}\label{E3.1}
\begin{split}
 &\int_{\R^{3}}\phi(x-x_{0})\Big(|u|^{2}+|F|^{2}\Big)(x,t)+2\int_{0}^{t}\int_{\R^{3}}\varphi(x-x_{0})\Big(|\nabla u|^{2}+|\nabla F|^{2}\big)\\
&\leq\int_{0}^{t}\int_{\R^{3}}\Bigg\{\Delta \varphi(x-x_{0})\Big(|u|^{2}+|F|^{2}\Big)
-2\sigma u\otimes \nabla\varphi\cdot FF^{t}\\
&\quad+\Big(\sigma(|u|^{2}+|F|^{2})+2p\Big)u\cdot\nabla\varphi\Bigg\}
+\int_{\R^{3}}\varphi(x-x_{0})\Big(|u_0|^2+|F_0|^2\Big)
\end{split}
\end{align}
holds for a.e. $t>0$, where $\varphi$ is a nonnegative, smooth cutoff function with $\varphi=1$ in $\B_{R}(x_0)$,
$\varphi=0$ outside $\B_{2R}(x_0)$, and $|\nabla \varphi|\leq \frac{C}{R}$. 

For $\lambda\leq1$, set
$$
A(\lambda)=\mbox{ess} \sup_{0\leq t<\lambda R^{2}}\sup_{x_{0}\in\R^{3}}
\int_{\R^{3}}\varphi(x-x_{0})\Big(|u|^{2}+|F|^{2}\Big)+
\sup_{x_{0}\in\R^{3}}\int_{0}^{\lambda R^{2}}\int_{\R^{3}}2\varphi(x-x_{0})\Big(|\nabla u|^{2}+|\nabla F|^{2}\Big).
$$
From the  multiplicative inequality
$$
\|u\|_{L^{3}(\B_{R})}\leq C\Big(\|\nabla u\|_{L^{2}(\B_{R})}^{\frac{1}{2}}\| u\|_{L^{2}(\B_{R})}^{\frac{1}{2}}+R^{-\frac{1}{2}}\| u\|_{L^{2}(\B_{R})}\Big),
$$
we have that for $\lambda\le 1$, 
$$
\sup_{x_{0}\in\R^{3}}\int_{0}^{\lambda R^{2}}\int_{\B_{R}(x_{0})}\Big(| u|^{3}+|F|^{3}\Big)\leq CA(\lambda)^{\frac{3}{2}}R^{\frac{1}{2}}\lambda^{\frac{1}{4}}.
$$
Let $ k(x)$ denote the kernel of $\Delta^{-1}\mbox{div}^2$ and define $p(x,t)$ by
\begin{align*}
p(x,t)=&-\Delta^{-1}\mbox{div}^2\Big[(u\otimes u+FF^t)\psi\Big]\\
&-\int_{\R^{3}}\Big[k(x-y)-k(x_{0}-y)\Big]\Big[(u\otimes u+FF^t)(1-\psi)\Big](y)dy+p(t)\\
=&I_{1}(x)+I_{2}(x),
\end{align*}
where $\psi\in C_0^\infty(\B_{8R}(x_0))$ is such that $\psi|_{\B_{4R}(x_{0})}=1, 0\leq\psi\leq 1$, and
$|\nabla \psi|\leq\frac{C}{R}$.

By the $L^{p}$ estimates, we have
\begin{align*}
\|I_{1}\|_{L^{\frac{3}{2}}(\B_{2R}(x_{0})\times(0,\lambda R^{2}))}&\leq C\Big(\| u\|^{2}_{L^{3}(\B_{8R}(x_{0})\times(0,\lambda R^{2}))}
+\| F\|^{2}_{L^{3}(\B_{8R}(x_{0})\times(0,\lambda R^{2}))}\Big)\\
&\leq C\lambda^{\frac{1}{6}}A(\lambda)R^{\frac{1}{3}}.
\end{align*}
On the other hand, since $k(x-y)-k(x_{0}-y)$ satisfies
$$
|k(x-y)-k(x_{0}-y)|\leq \frac{CR}{|x_{0}-y|^{4}}\ \ \mbox{for}\ |x_{0}-y|\geq 4R,\ |x-x_{0}|\leq 2R,
$$
we can easily obtain that for $x\in \B_{2R}(x_0)$, 
\begin{align*}
I_{2}(x)&\leq CR\int_{|y-x_{0}|\geq 4R}\frac{1}{|x_{0}-y|^{4}}(|u|^{2}+|F|^{2})(y)dy\\
&=CR\sum_{n=4}^{\infty}\int_{nR\leq|y-x_{0}|\leq (n+1)R}\frac{1}{|x_{0}-y|^{4}}(|u|^{2}+|F|^{2})(y)dy\\
&\leq CR^{-3}A(\lambda)\sum_{n=4}^{\infty}\frac{1}{n^{2}}\leq CR^{-3}A(\lambda).
\end{align*}
Combining the estimates of $I_{1}$ with $I_{2}$, we immediately have
\begin{equation}\label{E3.2}
\|p(x,t)-p(t)\|_{L^{\frac{3}{2}}(\B_{2R}(x_{0})\times(0,\lambda R^{2}))}\leq \lambda^{\frac{1}{6}}A(\lambda)R^{\frac{1}{3}}\ \ \mbox{for}\ \lambda\leq 1.
\end{equation}
Combining (\ref{E3.1}) with (\ref{E3.2}), and using H\"{o}lder inequality, we have
\begin{equation}\label{E3.3}
A(\lambda)\leq \alpha+CA(\lambda)\lambda+CA^{\frac{3}{2}}(\lambda)\lambda^{\frac{1}{4}}R^{-\frac{1}{2}}.
\end{equation}
Note that $A(\lambda)$ is a bounded, continuous function of $\lambda$ and $A(0)=0$.
By choosing $C_{1}$ sufficiently small, we obtain $A(\lambda)<C_2\alpha$ by the usual ``continuation in $\lambda$''
argument \cite{LR0, LR}. Finally, by (\ref{E3.2}), we have
$$
 \sup_{x_{0}\in\R^{3}}\int_{0}^{\lambda R^{2}}\int_{\B_{R}(x_{0})}|p(x,t)-p(t)|^{\frac{3}{2}}\le C_{2}\alpha^{\frac{3}{2}}R^{\frac{1}{2}}.
$$
This completes the proof.
\end{proof}

The following result plays a crucial role in the proof of Theorem \ref{T1.1}.
\begin{Lemma}\label{L3.2} 
Suppose $ u_0\in L^2_{\rm{loc}}(\R^3,\R^3), F_0\in L^2_{\rm{loc}}(\R^3,\R^{3\times 3})$,
with $\mbox{div} u_{0}=0$ and $\mbox{div} F_{0}=0$, satisfy
$$\|u_{0}\|_{L^{2}_2(\R^{3})}^{2}+\|F_{0}\|_{L^{2}_{2}(\R^{3})}^{2}\leq A_{1}<\infty,$$
and
$$\|u_{0}\|_{C^{\gamma}(\B_{2})}+\|F_{0}\|_{C^{\gamma}(\B_{2})}\leq A_{2}<\infty,
$$ for some $\gamma\in (0,1)$. 
Then there exists
$T=T(A_{1},A_{2},\gamma)$ such that  any local weak solution $(u, F)\in \mathcal{N} (u_{0},F_{0})$ 
of (1.3)$_\sigma$, for $\sigma \in [0,1]$, satisfies
$(u, F)\in C_{par}^{\gamma}(\overline{\B}_{\frac{1}{4}}\times[0,T])$, and
\begin{equation}\label{E3.2.2}
\big\|u\big\|_{ C_{par}^{\gamma}(\overline{\B}_{\frac{1}{4}}\times[0,T])}+\big\|F\big\|_{ C_{par}^{\gamma}(\overline{\B}_{\frac{1}{4}}\times[0,T])}\leq
 C(A_{1},A_{2},\gamma).
 \end{equation}
\end{Lemma}

\begin{proof} Let us decompose $u_{0}=u_{0}^{1}+u_{0}^{2}, F_{0}=F_{0}^{1}+F_{0}^{2}$, with
$\mbox{div} u_{0}^{1}=\mbox{div} F_{0}^{1}=0$, such that
$$\begin{cases} u_{0}^{1}|_{\B_{\frac{4}{3}}}=u_{0},\ \  F_{0}^{1}|_{\B_{\frac{4}{3}}}=F_{0},\\
\mbox{supp}\ u_{0}^{1},\ \ \mbox{supp}\ F_{0}^{1}\Subset \B_{2},\\
\|u_{0}^1\|_{L^{2}_2(\R^{3})}^{2}+\|F_{0}^1\|_{L^{2}_{2}(\R^{3})}^{2}\leq C A_{1},\\
\|u_{0}^{1}\|_{C^{\gamma}(\R^{3})}+\|F_{0}^{1}\|_{C^{\gamma}(\R^{3})}\leq CA_{2}. 
\end{cases}
$$
Let $(a, M)$ be the mild solution to the system (1.3)$_\sigma$, with $(u_{0}^{1},F_{0}^{1})$ as the initial data,
in $\R^{3}\times[0,T_{1})$ for some $T_1=T_1(A_1,A_2)>0$.
Then we have  that for any $5\le p<+\infty$ (see the appendix below),
\begin{equation}\label{E3.3.1}
\big\|(a,M)\big\|_{L^{p}(\R^{3}\times(0,T_{1}))}\leq C(A_1, A_{2}).
\end{equation}
Moreover, by the local energy estimate for $(a, M)$, we have
\begin{equation}\label{E3.4}
\mbox{ess} \sup_{0\leq t<T_{1}}\int_{\B_{2}(x_{0})}\Big(|a|^{2}+|M|^{2}\Big)+
2\int_{0}^{T_{1}}\int_{\B_{2}(x_{0})}\Big(|\nabla a|^{2}+|\nabla M|^{2}\Big)\leq C(A_{1},A_{2}),
\end{equation}
for any $x_{0}\in \R^{3}$. 

Write $u=a+v, F=M+H$, we can verify that $(v,H,q)$ is a weak solution of the perturbed system \eqref{E2.1}
in $\R^{3}\times [0,T_{1})$, that satisfies the local energy inequality (\ref{E2.2+}). Note also that
 $$
 \lim_{t\to0^{+}}\Big(\|v(\cdot,t)\|_{L^{2}(\B_{\frac43})}+\|H(\cdot,t)\|_{L^{2}(\B_{\frac43})}\Big)=0.
 $$
 Combining this with the local energy inequality for $(v,H,q)$, we obtain that for any $0<t\le T_1$,
\begin{align}\label{E3.6}
\nonumber&\int_{\B_{\frac43}}\varphi\Big(|v|^{2}+|H|^{2}\Big)+2\int_{0}^{t}\int_{\B_{\frac43}}\varphi\Big(|\nabla v|^{2}+|\nabla H|^{2}\Big)\leq\\
\nonumber&\int_{0}^{t}\int_{\B_{\frac43}}\Big\{\Delta\varphi(|v|^{2}+|H|^{2})-2\sigma u\otimes\nabla \varphi\cdot HH^{t}\Big\}+
\int_{0}^{t}\int_{\B_{\frac43}}\Big\{\sigma (|v|^{2}+|H|^{2})(v+a)+2qv\Big\}\cdot \nabla \varphi \\
\nonumber&+2\sigma\int_{0}^{t}\int_{\B_{\frac43}}\Big(a\otimes v-MH^t-HM^t\Big)\cdot \Big(\nabla v\varphi+v\otimes\nabla \varphi\Big)
\\&+2\sigma\sum\limits_{i,j,k=1}^{3}\int_{0}^{t}\int_{\B_{\frac43}}\Big (v_k M_{ij}-a_kH_{ij}-v_iM_{kj}\Big)
\cdot \Big (\nabla_k H_{ij}\phi+H_{ij}\nabla_k \phi\Big),
\end{align}
where $\phi\in C_{0}^{\infty}(\B_{\frac43})$ is such that $\phi|_{\B_{1}}\equiv 1, \phi\geq 0$.

By Lemma \ref{L3.1} and (\ref{E3.4}), we conclude that there exists $0<T_{2}=T_2(A_{1},A_{2})\le T_1$
such that
\begin{align}\label{E3.4.1}
\mbox{ess}\sup_{0\leq t<T_{2}}\int_{\B_{2}}(|v|^{2}+|H|^{2})+
\int_{0}^{T_{2}}\int_{\B_{2}}\big[|\nabla v|^{2}+|\nabla H|^{2}
+|q(x,t)-q(t)|^{\frac{3}{2}}\big]
\le C(A_{1},A_{2}).
\end{align}
This, by the interpolation inequalities, implies 
$$
\Big(\int_{0}^{T_{2}}\int_{\B_{2}}\big[|v|^{\frac{10}3}+|H|^{\frac{10}3}]\Big)^{\frac3{10}}\leq C(A_{1},A_{2}).
$$
Thus by (\ref{E3.6}) and H\"older's inequality we obtain that
\begin{align}\label{E3.6.1}
\sup_{0\le\tau\le t}\int_{\B_{1}}\Big(|v|^{2}+|H|^{2}\Big)(x,\tau)+
2\int_{0}^{t}\int_{\B_{1}}\Big(|\nabla v|^{2}+|\nabla H|^{2}\Big) 
\leq C(A_{1},A_{2})t^{\frac{1}{30}},
\end{align}
for $0<t\le T_{2}$. Here we have applied (\ref{E3.3.1}) with $p=6$. 
Since $q$ solves 
\begin{equation}\label{E3.5.1}
\Delta q=-\sigma{\rm{div}}^2\big[v\otimes a+a\otimes v+v\otimes v
-HH^t-MH^t-HM^t\big].
\end{equation}
From the standard $L^p$-estimate, (\ref{E3.4.1}),  and (\ref{E3.6.1}), we see that $q\in L_{\rm{loc}}^{\frac{5}{3}}(\B_2\times [0, T_2])$ is bounded. Hence
\begin{equation}\label{E3.6.2}
\Big(\int_{0}^{t}\int_{\B_{2}}|q|^{\frac{3}{2}}\Big)^{\frac{2}{3}}\leq C(\alpha,A_{1},A_{2})t^{\frac{1}{15}}.
\end{equation}
Now for $t_{0}>0$ fixed,  extend $v,H,q$ and $a, M$ to $\B_{1}\times(-1+t_{0},t_{0}]$ by letting
$$v=0, H=0, q=0, a=0, M=0\ \ {\rm{in}}\ \ B_{1}\times(-1+t_{0},0].$$
Note that
$$
 \lim_{t\to0^{+}}\left(\|v(\cdot,t)\|_{L^{2}(\B_{4/3}(0))}+\|H(\cdot,t)\|_{L^{2}(\B_{4/3}(0))}\right)=0.
 $$
We can check that $(v,H,q)$ is a suitable weak solution to the system (1.3)$_\sigma$ in $\B_{1}\times[-1+t_{0},t_{0})$.
By choosing $t_{0}=t_{0}(\alpha,A_{1},A_{2})$ sufficiently small and using (\ref{E3.6.1}) and (\ref{E3.6.2}), we can apply Theorem \ref{Th2.1} to conclude that  $(v,H)$ is H\"{o}lder continuous in
$\B_{\frac{1}{2}}\times [0,t_{0}]$ for some $0<\beta\le\gamma$, with
$$
\big\|(v, H)\big\|_{C_{par}^{\beta}(\B_{\frac12}\times[0,t_{0}])}\leq C(\alpha,A_{1},A_{2}).
$$
By a standard bootstrapping argument, we can improve the exponent $\beta$ to equal to $\gamma$. 
Since $a,M\in C_{par}^{\gamma}(\B_{\frac12}\times[0,t_{0}])$, 
we conclude that  $u,F\in  C_{par}^{\gamma}(\B_{\frac12}\times[0,t_{0}])$ and (\ref{E3.2.2}) holds.
This completes the proof.
\end{proof}

\subsection{A-priori estimate for the self-similar solution}
This subsection is devoted to establish a priori estimates for forward self-similar solutions
of (1.3)$_\sigma$ for $\sigma\in [0,1]$.

\begin{Lemma}\label{L3.3} For $\gamma\in (0,1)$, 
let $u_{0}\in C_{\rm{loc}}^{\gamma}\Big(\R^{3}\setminus \{0\},\R^3\Big)$ and
$F_0\in  C_{\rm{loc}}^{\gamma}\Big(\R^{3}\setminus \{0\},\R^{3\times 3}\Big)$,
with $\mbox{div}u_{0}=0$ and $\mbox{div}F_{0}=0$, be homogeneous of degree $-1$. 
Assume (\ref{E1.4}) holds.  Let $(u, F)\in \mathcal{N} (u_{0},F_{0})$ be a forward self-similar solution
of (1.3)$_\sigma$ 
for $\sigma\in [0,1]$. 
Then there exists $C>0$, independent of $\sigma$, such that
$\mathcal U(\cdot):=u(\cdot,1)$ and $\mathcal F(\cdot):=F(\cdot,1)$ satisfy
\begin{equation}\label{S3.1}
\big|\mathcal U(x)-e^{\Delta}u_{0}(x)\big|+\big|\mathcal F(x)-e^{\Delta}F_{0}(x)\big|
\leq\frac{C}{(1+|x|^{2})^{\frac{1+\gamma}2}},
\ \ \forall\ x \in \R^{3}.
\end{equation}
\end{Lemma}
\begin{proof} First, by Lemma \ref{L3.2},  there exists $T_{1}=T_{1}(\gamma, C_{*})$, where
$C_*$ is the constant in (\ref{E1.4}),  such that for any $x_{0}\in \R^{3}$ with $|x_{0}|=2$,
$$
\Big\|\big(u(x,t)-e^{t\Delta}u_{0}(x), F(x,t)-e^{t\Delta}F_{0}(x)\big) \Big\|_{C_{par}^{\gamma}({\B}_{\frac{1}{9}}(x_{0})\times[0,T_{1}])}\leq C(\gamma, C_{*}),
$$
where we have used the fact that both $e^{t\Delta}u_{0}$ and $e^{t\Delta}F_{0}$ satisfy 
the heat equation so that  the same H\"{o}lder estimates hold. Since
$$u(\cdot,t)-e^{t\Delta}u_{0}\big|_{t=0}=0, \ F(\cdot,t)-e^{t\Delta}F_{0}\big|_{t=0}=0, $$ 
it follows that 
$$
\big|u(x,t)-e^{t\Delta}u_{0}(x)\big|+\big|F(x,t)-e^{t\Delta}F_{0}(x)\big|
\leq C(\gamma, C_{*}) t^{\frac{\gamma}{2}}, \forall x\in {\B}_{\frac{1}{9}}(x_{0}),\  0\leq t\leq T_{1}.
$$
Since $u, F, e^{t\Delta} u_0, e^{t\Delta F_0}$ are all self-similar, this immediately implies that 
$$
|\mathcal U(y)-e^{\Delta}u_{0}(y)|+|\mathcal F(y)-e^{\Delta}F_{0}(y)|\leq\frac{C}{(1+|y|^{2})^{\frac{1+\gamma}{2}}}
$$
holds  for any  $|y|>\frac{2}{\sqrt{T_{1}}}.$

On the other hand, since $u_0, F_0$ are homogeneous of degree $-1$, it follows that 
$$\alpha:=\|u_{0}\|_{L^{2}_1(\R^3)}^2+\|F_{0}\|_{L^{2}_1(\R^3)}^2<+\infty.$$
Hence  by Lemma \ref{L3.1} there exists $T_1=T_1(\alpha)>0$ such that 
\begin{equation}\label{S3.2}
\sup_{0<t<T_1}\int_{\B_1}(|u(x,t)|^2+|F(x,t)|^2)
+\int_{0}^{T_1}\int_{\B_1}(|\nabla u(x,t)|^2+|\nabla F(x,t)|^2)\leq C\alpha.
\end{equation}
Since $u,F$ are self-similar, it follows from direct calculations and (\ref{S3.2}) that for a fixed $t_*<T_1$, to be determined later,
\begin{align*}
&\sqrt{t_*}\int_{\B_{\frac{1}{\sqrt{t_*}}}}\Big(|\mathcal U|^2+|\mathcal F|^2+|\nabla \mathcal U|^2+|\nabla \mathcal F|^2\Big)\\
&\leq C\int_{\B_1}\Big(|u(x,t_*)|^2+|F(x,t_*)|^2\Big)+
C\int_0^{t_*}\int_{\B_1}\Big(|\nabla u(x,t)|^2+|\nabla F(x,t)|^2\Big)\\
&\le C(\alpha).
\end{align*}
This implies, by choosing $t_{*}=\frac{T_1}{16}$, that
$$
\int_{\B_{\frac{4}{\sqrt{T_{1}}}}}\Big(|\nabla \mathcal U|^{2}+|\nabla \mathcal F|^{2} 
+|\mathcal U|^{2}+|\mathcal F|^{2}\Big)\leq C(T_1).
$$
Since $(\mathcal U,\mathcal F)$ solves the  elliptic system 
\begin{align}\label{S3.3}
\left\{
\begin{array}{llll}
-\Delta \mathcal U-\frac{1}{2}\mathcal U-\frac{1}{2}(x\cdot\nabla)\mathcal U
+\sigma (\mathcal U\cdot\nabla) \mathcal U-\sigma (\mathcal F_{l}\cdot\nabla)\mathcal F_{l}+\nabla \mathcal P=0,\\\\
{\rm{div}} \mathcal U=0,\ {\rm{div}} \mathcal F=0,\\\\
-\Delta \mathcal F_{j}-\frac{1}{2}\mathcal F_{j}-\frac{1}{2}(x\cdot \nabla) \mathcal F_{j}
+\sigma (\mathcal U\cdot\nabla)\mathcal F_{j}-\sigma (\mathcal F_{j}\cdot\nabla)\mathcal U=0, \ j=1, 2, 3,
\end{array}
\right.
\end{align}
in $\B_{\frac{4}{\sqrt{T_1}}}$,  we have, by the standard regularity theory of elliptic system in dimension three,
that for any $k\ge 0$, 
$$
\big\|\mathcal U\big\|_{C^{k}(\B_{\frac{2}{\sqrt{T_1}}})}+\big\|\mathcal F\big\|_{C^{k}(\B_{\frac{2}{\sqrt{T_{1}}}})}\leq C(k,\gamma, C_*).
$$
Combining these two estimates yields (\ref{S3.1}). This completes the proof. 
\end{proof}

\subsection{Proof of main result}
In this subsection, we will prove Theorem \ref{T1.1} by the Leray-Schauder fixed point theorem. We begin with some well-known results
on the non-stationary Stokes system in $\R^{3}$ with a force tensor $f=(f_{ij})$
\begin{equation}\label{E3.7}
\begin{cases}
\partial_{t}v-\Delta v+\nabla p=\nabla\cdot f, \\
\nabla\cdot v=0,\\
v(\cdot,0)=0,
\end{cases} \ {\rm{in}}\ \R^3\times (0,\infty),
\end{equation}
where $(\nabla\cdot f)_{j}=\partial_{k}f_{kj}$ for $1\le j\le 3$. If $f$ has sufficient decay near infinity, the solution of (\ref{E3.7}) is given by $v=\Phi f$, where
\begin{equation}\label{E3.8}
(\Phi f)_{i}(x,t)=\int_{0}^{t}\int_{\R^{3}}\partial_{x_{k}}S_{ij}(x-y,t-s)f_{kj}(y,s)dyds, \ i=1,2,3,
\end{equation}
and $S=(S_{ij})$ is the fundamental solution of the non-stationary Stokes system in $\R^{3}$ (see \cite{O,S}), called as
the Oseen kernel, given by
$$
S_{ij}(x,t)=\Big(-\delta_{i,j}\Delta+\frac{\partial^{2}}{\partial x_{i}\partial x_{j}}\Big)\frac{1}{4\pi}\int_{\R^{3}}\frac{\Gamma(y,t)}{|x-y|}dy,\ (x,t)\in\R^3\times (0,\infty), 
$$
where $\displaystyle\Gamma(x,t)=(4\pi t)^{-\frac{n}{2}}e^{-\frac{x^{2}}{4t}}$ is the heat kernel. It is well-known (see \cite{S}) that
$$
|D_{x}^{l}\partial_{t}^{k}S(t,x)|\leq C_{k,l}\Big(|x|+\sqrt{t}\Big)^{-3-l-2k},\ \ l,k\geq 0,
$$
where $D_{x}^{l}$ denotes the $l$ order derivatives with respect to the variable $x$.

Now we recall some useful integral estimates which play key roles in the proof of our main results.



\begin{Lemma}\label{C3.1}{\rm\cite{JS,T}}
For $0\le\alpha<1$, suppose $\displaystyle |f(x,t)|\leq \frac{1}{t}\bigg(\frac{\sqrt{t}}{|x|+\sqrt{t}}\bigg)^{2+\alpha}$ in $\R^3\times (0,\infty)$. Then
$$
\big|\Phi f (x,t)\big|\lesssim \frac{1}{\sqrt{t}}\bigg(\frac{\sqrt{t}}{|x|+\sqrt{t}}\bigg)^{2+\alpha},\ \ \forall (x,t)\in \R^3\times (0,\infty).
$$
\end{Lemma}
\vskip 0.2in
The following Lemma shows that $\Phi f (x,t)$ is H\"{o}lder continuous in space and time, provided that $f$ has sufficient decay near infinity.

\begin{Lemma}\label{L3.5}{\rm{\cite{T}}}
Suppose $\displaystyle |f(x,t)|\leq (|x|+\sqrt{t})^{-2}$ in $\R^3\times (0,\infty)$. Then $\Phi f$ is locally H\"{o}lder continuous in $x$ and $t$ with any exponent
$0<\theta<1$. \qed
\end{Lemma}

Finally we recall the following Liouville type property. Denote 
$$\langle x\rangle=\sqrt{1+|x|^2}\ \ {\rm{for}}\ \  x\in\R^3.$$
\begin{Lemma}\label{L3.6} \cite{T}
If $v(x,t):\R^3\times (0,\infty)\to\R^{3}$ is a solution of the Stokes equation (\ref{E3.7}), with $f\equiv 0$, for some distribution $p$,
and satisfies 
$$|v(x,t)|\leq {Ct^{-\frac{1}{2}}}\langle\frac{x}{\sqrt{t}}\rangle^{-(1+\gamma)} \ \ \mbox{for some}\ \ 0<\gamma<1,$$
then $v\equiv0$ in $\R^3\times (0,\infty)$.
\end{Lemma}

\noindent{\bf Proof of Theorem \ref{T1.1}}. For $(x,t)\in\R^3\times (0,\infty)$, set
$$
U^{0}(x,t):=(e^{t\Delta}u_{0})(x),\ \  G^{0}(x,t):=(e^{t\Delta}F_{0})(x).
$$
From the assumptions on $u_{0}$ and $F_{0},$ we have that $ U_{0}$ and $G_{0}$ are self-similar solutions to the heat equation, and satisfy
$$
|U^{0}(x,t)|+|G^{0}(x,t)|\leq C\big(|x|^{2}+t\big)^{-\frac12}, \  (x,t)\in\R^3\times (0,\infty).
$$
For $\sigma\in [0,1]$, we look for a family of self-similar solutions ($u_\sigma, F_\sigma $)
to the system (1.3)$_\sigma$ and (\ref{IC1.1}) of the form
\begin{align*}
&u_\sigma(x,t)=U^{0}(x,t)+v_\sigma(x,t),\  u_\sigma(x,0)=u_{0}(x);\\&  
F_\sigma(x,t)=G^{0}(x,t)+H_\sigma(x,t),\ F_\sigma(x,0)=F_{0}(x).
\end{align*}
In particular, $v_\sigma$ and $H_\sigma$ are also self-similar, i.e.,
\begin{align}\label{E3.14}
\begin{split}
&v_\sigma(x,t)=\frac{1}{\sqrt{t}}v_\sigma\Big(\frac{x}{\sqrt{t}} ,1\Big)=\frac{1}{\sqrt{t}}\widehat{v}_\sigma\Big(\frac{x}{\sqrt{t}}\Big),\\&
 H_\sigma(x,t)=\frac{1}{\sqrt{t}}H_\sigma\Big(\frac{x}{\sqrt{t}} ,1\Big)
 =\frac{1}{\sqrt{t}}\widehat{H}_\sigma\Big(\frac{x}{\sqrt{t}}\Big).
\end{split}
\end{align}
For $x\in\R^3$ and $t\in (0,\infty)$, set
$$q_\sigma(x,t)= \frac{1}{t}\widehat{q}_\sigma\big(\frac{x}{\sqrt{t}}\big),$$
where
$$
\widehat{q}_\sigma(x)=\big(U^{0}(x,1)+\widehat{v}_\sigma(x)\big)\otimes \big(U^{0}(x,1)+\widehat{v}_\sigma(x)\big)
+\big(G^{0}(x,1)+\widehat{H}_{\sigma}(x)\big)\big(G^{0}(x,1)+\widehat{H}_{\sigma}(x)\big)^t.
$$
Similarly, set
\begin{align*}
Q^\sigma_j(x,t)=\frac{1}{t}\widehat{Q}_{j}^\sigma(\frac{x}{\sqrt{t}}),\ \ 1\le j\le 3,
\end{align*}
where
\begin{align*}
\widehat{Q}_{j}^\sigma(x)=&-(U^{0}(x,1)+\widehat{v}_\sigma(x))\otimes(G^0_j(x,1)+(\widehat{H}_\sigma)_{j}(x))\\
&+(G^0_j(x,1)+(\widehat{H}_\sigma)_{j}(x))\otimes(U^{0}(x,1)+\widehat{v}_\sigma(x)), \ j=1, 2, 3.
\end{align*}
It is readily seen that $(v_\sigma, H_\sigma)$, for $0\le\sigma\le 1$, solves the following equations:
\begin{align}\label{E3.9}
\left\{\displaystyle
\begin{array}{llll}
\partial_{t}v_\sigma-\Delta v_\sigma+\nabla p_\sigma
=\sigma\nabla\cdot q_\sigma(x,t)\ \Big(=\sigma\nabla\cdot\big(\frac{1}{t}\widehat{q}_\sigma(\frac{x}{\sqrt{t}})\big)\Big), \\
\displaystyle\partial_{t}H_\sigma-\Delta H_\sigma=\sigma\sum_{j=1}^3\nabla\cdot Q^\sigma_j(x,t)\ 
\Big(=\sigma\sum_{j=1}^3\nabla\cdot\big(\frac{1}{t}\widehat{Q}^\sigma_j(\frac{x}{\sqrt{t}})\big)\Big), \\
\mbox{div} v_\sigma=\mbox{div} H_\sigma=0,\\
v_\sigma(x,0)=0, H_\sigma(x,0)=0.
\end{array}
\right.
\end{align}
Inspired by \cite{JS, T}, we aim to find a self-similar solution to the system (\ref{E1.2+}) by the Leray-Schauder fixed
point theory, which is slightly different from \cite{JS}. To do it, define, for $\gamma>0$, 
the Banach space $X_{\gamma}$ by
$$
X_\gamma=\Big\{u\in C(\R^{3},\R^{3})\ \big|\ \mbox{div}\ u=0,\ \|u\|_{X_\gamma}<\infty\Big\},
$$
where
$$
\big\|u\big\|_{X_\gamma}:=\sup_{x\in\R^{3}}\ \langle x \rangle^{1+\gamma}|u(x)|.
$$
Set $X_\gamma^{4}=\Big\{(u_1,\cdots,u_4)\ \big|\ u_i\in X_\gamma, i=1,2,3,4\Big\}$, which is equipped with the norm 
$$
\big\|u\big\|_{X_\gamma^{4}}=\sum_{j=1}^{4}\big\|u_{j}\big\|_{X_\gamma}.
$$
Observe that for any $\big(\ \widehat{v}_\sigma,\widehat{H}_\sigma\ \big)\in X_\gamma^4$, we have that
$\big(\ \widehat{q}_\sigma, \widehat{Q}^\sigma_1, \widehat{Q}^\sigma_2, \widehat{Q}^\sigma_3\big)\in X^4_{1+2\gamma}$, and  
\begin{align*}
\big\|(\ \widehat{q}_\sigma, \widehat{Q}^\sigma_1,\widehat{Q}^\sigma_2, \widehat{Q}^\sigma_3)\big\|_{X^4_{1+2\gamma}}
\leq C\Big[1+\big\|(\ \widehat{v}_\sigma, \widehat{H}_\sigma)\big\|_{X^4_{\gamma}}^2\Big].
 \end{align*}
In particular, we have that
\begin{align}\label{E3.10}
|q_\sigma(x,t)|+\sum_{j=1}^3|Q^\sigma_j(x,t)|\leq \frac{C}{|x|^{2}+t}\Big[1+\big\|(\widehat{v}_\sigma,\widehat{H}_\sigma)\big\|_{X_\gamma^4}^2\Big],\ \mbox{for}\ \ (x,t)\in \R^{3}\times (0,\infty).
\end{align}
By Lemma \ref{C3.1} and Lemma \ref{L3.6},  we conclude that for any given $\sigma\in [0,1]$ and $\big(\ \widehat{v}_\sigma, \widehat{H}_\sigma\ \big)\in X_\gamma^4$,
the system (\ref{E3.9}) has a unique, self-similar solution 
$$(v_\sigma, H_\sigma)=\Big(K_0(\widehat{v}_\sigma, \widehat{H}_\sigma,\sigma), 
K_1(\widehat{v}_\sigma, \widehat{H}_\sigma,\sigma)\Big),$$
where
$$
v_\sigma(\cdot,t)=K_0(\widehat{v}_\sigma, \widehat{H}_\sigma,\sigma)(\cdot,t)
=\sigma\int_{0}^{t}e^{-\Delta(t-\tau)}\mathbb P\nabla\cdot\Big(\tau^{-1}\widehat{q}_\sigma(\frac{\cdot}{\sqrt{\tau}})\Big)d\tau, 
=\Phi(\sigma q_\sigma)(t), 
$$
and
$$
H_\sigma(\cdot,t)=K_1(\widehat{v}_\sigma, \widehat{H}_\sigma,\sigma)(\cdot,t)
=\sigma\sum_{j=1}^3\int_{0}^{t}e^{-\Delta(t-\tau)}\mathbb P\nabla\cdot\Big(\tau^{-1}\widehat{Q}^\sigma_j(\frac{\cdot}{\sqrt{\tau}})\Big)d\tau
=\sum_{j=1}^3\Phi(\sigma Q_j^\sigma)(t),
$$
where $\Phi$ is given by (\ref{E3.8}), and $\mathbb P$ is the Leray projection operator. 
Moreover, $(v_\sigma,H_\sigma)$ satisfies the estimate
\begin{equation}\label{E3.11}
 |v_\sigma(x,t)|+|H_\sigma(x,t)|\leq Ct^{-\frac12}\langle \frac{x}{\sqrt t}\rangle^{-2} 
 \Big[1+\big\|(\widehat{v}_\sigma,\widehat{H}_\sigma)\big\|_{X_\gamma^4}^2\Big], \ \mbox{for}\ \ (x,t)\in \R^{3}\times (0,\infty),
\end{equation}
Since $(v_\sigma, H_\sigma)$ is self-similar, we can write
$$(v_\sigma, M_\sigma)(x,t)=\frac{1}{\sqrt{t}}\Big(\widetilde{v}_\sigma, \widetilde{H}_\sigma\Big)\big(\frac{x}{\sqrt{t}}\big), \ (x,t)\in\R^3\times (0,\infty).$$
It follows directly from (\ref{E3.11}) that $\big(\widetilde{v}_\sigma, \widetilde{H}_\sigma\big)\in X^4_{1}$. Furthermore, we have that
\begin{equation}\label{E3.12}
\Big(\widetilde{v}_\sigma, \widetilde{H}_\sigma\Big)(x)=\Big(K_0(\widehat{v}_\sigma,\widehat{H}_\sigma,\sigma), K_1(\widehat{v}_\sigma, \widehat{H}_\sigma, \sigma)\Big)(x,1), \ x\in \R^{3}.
\end{equation}
Define the operator $\mathbb T$ on  $X_{\gamma}^4\times [0,1]$ by
$$\mathbb T(\widehat{v}_\sigma,\widehat{H}_\sigma; \sigma)(x)=\Big(K_0(\widehat{v}_\sigma,\widehat{H}_\sigma,\sigma), K_1(\widehat{v}_\sigma, \widehat{H}_\sigma, \sigma)\Big)(x,1), \ x\in\R^3.$$
From (\ref{E3.12}), we have
\begin{equation}\label{E3.13}
\Big\|\mathbb T(\widehat{v}_\sigma,\widehat{H}_\sigma; \sigma)\Big\|_{X^4_1}\le 
C\Big\{1+\Big\|(\widehat{v}_\sigma, \widehat{H}_\sigma)\Big\|_{X_\gamma^4}^2\Big\}.
\end{equation}
For $0<\gamma<1$, since $X_1^4$ is compactly embedded into $X_{\gamma}^4$ , we have from (\ref{E3.13}) that $\mathbb T$ is a compact operator from
$X_{\gamma}^4\times [0,1]$ to $X_{\gamma}^4$. 
It is clear that (\ref{E3.12}) can be rewritten as 
\begin{equation}\label{E3.15}
\Big(\widetilde{v}_\sigma, \widetilde{H}_\sigma\Big)=\mathbb T(\widehat{v}_\sigma,\widehat{H}_\sigma; \sigma).
\end{equation}
Now we want to apply the Leray-Schauder fixed point theorem to show that 
$\mathbb T(\cdot, \cdot;  1)$ has a fixed point $(\widehat{v}_1,\widehat{H}_1)\in X_\gamma^4$, i.e., 
$$
\Big(\widehat{v}_1, \widehat{H}_1\Big)=\mathbb T(\widehat{v}_1,\widehat{H}_1; 1).
$$
In fact, it follows from the definition of $\mathbb T$ and (\ref{E3.13}) that\\
(i) $\mathbb T(v, H; 0)=0$ for any $(v, H)\in X_\gamma^4$;\\
(ii) $\mathbb T: X_\gamma^4\times [0,1]\to X_\gamma^4$ is a compact operator.

Moreover,  Lemma \ref{L3.3} implies that there exists $C>0$, independent of $\sigma$, such that
the following holds:\\
(iii) if $(v_\sigma, H_\sigma)\in X_\gamma^4$ is a fixed point of $\mathbb T(\cdot,\cdot; \sigma)$, i.e.,
$$\Big({v}_\sigma, {H}_\sigma\Big)=\mathbb T(v_\sigma,H_\sigma; \sigma),$$
then
$$\big\|(v_\sigma, H_\sigma)\big\|_{X_\gamma^4}\le C.$$
Thus the Leray-Schouder fixed point theorem implies that  
there is  a fixed point $(v, H)\in X_\gamma^4 $ of $\mathbb T(\cdot, \cdot; 1)$.
It is readily seen that $u(x,t)=e^{t\Delta} u_0+\frac1{\sqrt{t}}v(\frac{x}{\sqrt{t}}),
F(x,t)=e^{t\Delta} F_0+\frac1{\sqrt{t}}H(\frac{x}{\sqrt{t}})$ solves (\ref{E1.2+}) and (\ref{IC1.1}), and
$(u,F)\in C^\infty(\R^3\times (0,\infty))\cap C^\gamma(\R^3\times [0,\infty)\setminus\{(0,0)\}))$.
This completes the proof.
\qed

\setcounter{section}{3}
\setcounter{equation}{0}
\section{Appendix}

In this appendix, we sketch a proof of the local existence of mild solutions to \eqref{E1.2+} or (1.3)$_\sigma$
for $0\le\sigma\le 1$.
For simplicity, we only consider \eqref{E1.2+} (or (1.3)$_\sigma$ for $\sigma=1$).

We begin by recalling a well-known fact on the Cauchy problem for the linear Stoke system (see, for example,
\cite{K}):
\begin{align}\label{E4.1}
\left\{
\begin{array}{llll}
&\partial_{t}u-\Delta u=\mbox{div}\ f-\nabla q \ &\  \mbox{in}\ \ Q_{T},\\
& \mbox{div}\ u=0 \ &\  \mbox{in}\ \ Q_{T},\\
&u(\cdot,0)=a(\cdot) \ &\  \mbox{in}\  \  \R^{3}.\\
\end{array}
\right.
\end{align}
Here $Q_{T}=\R^{3}\times (0,T)$.

\begin{Theorem}\label{T4.1}
For $m\geq 3$, suppose that
\begin{equation}\label{E4.2}
f\in L^{\frac{5m}6}(Q_{T})\cap L^{2}(Q_{T}),\ \ a\in L^{m}(\R^{3})\cap \mathbf {V}(\R^{3}).
\end{equation}
For any $T>0$, there is a pair of functions $(u,q)$  with the following properties:
\begin{equation}\label{E4.3}
\begin{cases} u\in C([0,T];L^{2})\cap L^{2}(0,T; {\bf V}(\R^3)),\\ 
\partial_{t}u\in L^{2}(0,T; ({\bf V}(\R^3))'),\\
u\in C([0,T];L^{m})\cap L^{\frac{5m}3}(Q_{T}),\\
q\in L^{\frac{5m}6}(Q_{T})\cap L^{2}(Q_{T}),
\end{cases}
\end{equation}
and ($u$, $q$) satisfy (\ref{E4.1}) in the sense of distributions, and 
\begin{equation*}\label{E4.6}
\|u(\cdot,t)-a(\cdot)\|_{L^m(\R^3)}\xrightarrow{t\rightarrow 0} 0.
\end{equation*}
Moreover, 
\begin{equation}\label{E4.7}
\big\|u\big\|_{L^\infty_tL^m_x(Q_T)}+\big\|u\big\|_{L^{\frac{5m}{3}}(Q_T)}
\leq C\Big(\|f\|_{L^{\frac{5m}{6}}(Q_T)}+\|a\|_{L^m(Q_T)}\Big).
\end{equation}
\end{Theorem}

Now we want to apply Theorem \ref{T4.2} to show the local existence of mild solutions to
\eqref{E1.1+}. 
\begin{Theorem}\label{T4.2} For $m\ge 3$, 
assume $u_0\in L^{m}(\R^3,\R^3)$ and $F_0\in L^m(\R^3,\R^{3\times 3})$ with
$\mbox{div} u_0=0$ and $\mbox{div} F_0=0$. Then there exist $T_*>0$, depending only on $(u_{0},F_{0})$,
and a triple of functions $u, F, q$ such that 
\begin{equation}\label{E4.14}
\begin{cases}
(u, F)\in C([0,T_*];L^{2}(\R^3,\R^3\times \R^{3\times 3})\cap L^{2}(0,T; {\bf V}(\R^3)),\\
(\partial_{t}u,\partial_{t}F)\in L^{2}(0,T_*; ({\bf V}(\R^3))'),\\
(u,F)\in C([0,T_*];L^{m}(\R^3))\cap L^{\frac{5m}3}(Q_{T_*}),\\
q\in L^{\frac{5m}6}(Q_{T_*})\cap L^{2}(Q_{T_*})\cap C([0,T_*];L^{\frac{m}2}(\R^3)),
\end{cases}
\end{equation}
and $(u,F, q)$ satisfy (\ref{E1.2+}) in the sense of distributions, and 
\begin{equation}\label{E4.17}
\big\|u(\cdot,t)-u_{0}(\cdot)\big\|_{L^m(\R^3)}+\big\|F(\cdot,t)-F_{0}(\cdot)\big\|_{L^m(\R^3)}\xrightarrow{t\rightarrow 0} 0.
\end{equation}
\end{Theorem}

\begin{proof}\ The proof is based on the standard successive iterations (see, for instance, \cite{K}). Here we
only sketch it.  Let
\begin{align}\label{EN4.9}
\begin{split}
&u^{1}(\cdot,t)=\Gamma(\cdot,t)\ast u_{0},\ \  F^{1}(\cdot,t)=\Gamma(\cdot,t)\ast F_{0},\\
&\kappa(T_{*})=\|u^{1}\|_{L^{\frac{5m}{3}}(Q_{T_{*}})}+\|F^{1}\|_{L^{\frac{5m}{3}}(Q_{T_{*}})},
\end{split}
\end{align}
and for $k\ge 1$, set
$$
u^{k+1}=w+u^{1}, F^{k+1}=G+F^{1},
$$
where $(w,G)$ solves
\begin{align}\label{E4.19}
\left\{
\begin{array}{llll}
\partial_{t}w-\Delta w=\nabla\cdot F^{k}(F^{k})^{t}-(u^{k}\cdot\nabla)u^{k}-\nabla q^{k};\\
\partial_{t}G-\Delta G=\nabla u^{k}F^{k}-(u^{k}\cdot\nabla)F^{k};\\
\nabla\cdot w=0, \nabla\cdot G=0;\\
w(x,0)=0, G(x,0)=0.
\end{array}
\right.
\end{align}
According to Theorem 4.1, we have the estimate
\begin{align*}
\|u^{k+1}-u^{1}\|_{L^{\frac{5m}{3}}(Q_{T_*})}+\|F^{k+1}-F^{1}\|_{L^{\frac{5m}{3}}(Q_{T_*})}
\leq C_{0}\Big(\|u^{k}\|_{L^{\frac{5m}{3}}(Q_{T_*})}+\|F^{k}\|_{L^{\frac{5m}{3}}(Q_{T_*})}\Big)^2.
\end{align*}
This estimate can be rewritten as
\begin{align}\label{E4.20}
\|u^{k+1}\|_{L^{\frac{5m}{3}}(Q_{T_*})}+\|F^{k+1}\|_{L^{\frac{5m}{3}}(Q_{T_*})}
&\leq C\Big(\|u^{k}\|_{L^{\frac{5m}{3}}(Q_{T_*})}
+\|F^{k}\|_{L^{\frac{5m}{3}}(Q_{T_*})}\Big)^{2}\nonumber\\
&\quad +\|u^{1}\|_{L^{\frac{5m}{3}}(Q_{T_*})}
+\|F^{1}\|_{L^{\frac{5m}{3}}(Q_{T_*})}.
\end{align}
We now want to show that $T_{*}$ can be chosen in such a way that
\begin{align}\label{E4.21}
\|u^{k+1}\|_{L^{\frac{5m}{3}}(Q_{T_*})}+\|F^{k+1}\|_{L^{\frac{5m}{3}}(Q_{T_*})}
\leq 2\kappa(T_{*})
\end{align}
for all $k\ge 1$.

We argue by induction on $k$. By virtue of \eqref{E4.20},
$$
\|u^{k+1}\|_{L^{\frac{5m}{3}}(Q_{T_*})}+\|F^{k+1}\|_{L^{\frac{5m}{3}}(Q_{T_*})}
\leq 4C_{0}\kappa^{2}(T_{*})+\kappa(T_{*}).
$$
Obviously, (\ref{E4.21}) holds if we choose $T_{*}$ such that
\begin{align}\label{E4.22}
\kappa(T_{*})<\frac{1}{4C_{0}}.
\end{align}
This is clearly possible, since both $u^1$ and $F^1$ solve the heat equation with initial condition $u_0$ and $F_0$.
From (\ref{E4.21}), we can apply the estimates of parabolic equations and  and  establish all the
statements of Theorem \ref{T4.2} by letting $(u, F, q)$ be the limit of $(u_k, F_k, q_k)$
as $k\rightarrow\infty$, except the continuity of $\|(u(t),F(t))\|_{L^m(\R^3)}$ with respect to $t$, which is a consequence of Theorem \ref{T4.1}. 
The continuity of $\|q(t)\|_{L^{\frac{m}2}(\R^3)}$,  as a function of $t$, follows from
the $L^p$-estimate of the pressure equation
$$
\Delta q=-\mbox{div}^2 \big(u\otimes u+FF^t\big),
$$
and the continuity of $\|(u(t),F(t))\|_{L^m(\R^3)}$ as function of $t$.
This completes the proof of Theorem \ref{T4.2}.
\end{proof}

\noindent{\bf Acknowledgements}. The first author is partially supported by
NSF of China 11201119, 11471099, and the International Cultivation of Henan Advanced Talents  
and Research Foundation of Henan University (yqpy20140043).
The second author is partially supported by the Fundamental Research Funds for the Central Universities,
SCUT 2015ZM183 and NSF of China 11571117. The third author is partially supported by NSF DMS
1522869.

\bigskip
\noindent
Baishun Lai: Institute of Contemporary Mathematics, Henan University, Kaifeng 475004, P.R. China

\noindent
{\sl E-mail address}: {\tt laibaishun@henu.edu.cn}

\medskip
\noindent
Junyu Lin: Department of Mathematics, South China  University of Technology, Guangzhou 510640, P. R. China.

\noindent
{\sl E-mail address}: {\tt scjylin@sctu.edu.cn}

\medskip
\noindent
Changyou Wang: Department of Mathematics, Purdue University, 150 N. University Street, West Lafayette, IN 47907, USA

\noindent
{\sl E-mail address}: {\tt wang2482@purdue.edu}

\end{document}